\newif\ifcomment
\commentfalse
\commenttrue

\documentclass[11pt]{scrartcl}
\usepackage{graphicx}
\usepackage{todonotes}
\usepackage{titling}
\usepackage[hyphens]{url}
\usepackage[defaultlines=2,all]{nowidow}

\newcommand\blfootnote[1]{%
  \begingroup
  \renewcommand\thefootnote{}\footnote{#1}%
  \addtocounter{footnote}{-1}%
  \endgroup
}
\usepackage{tikz,pgf}
\usetikzlibrary{calc,shapes}
\usetikzlibrary{hobby}
\usetikzlibrary{decorations.pathreplacing,arrows.meta}
\usetikzlibrary{positioning}
\usepackage{xspace,cite}
\usepackage{amsmath}
\usepackage{amsthm}
\usepackage{amssymb}
\usepackage{enumitem}
 \setlist[enumerate,1]{label=(\roman*)}
 \setlist[enumerate,2]{label=\alph*)}
 \setlist[enumerate,3]{label=\arabic*)}
\usepackage[basic]{complexity}
\usepackage{dsfont}
\usepackage{stmaryrd}
\usepackage{thmtools}
\usepackage{thm-restate}
\usepackage{hyperref}
\usepackage[nameinlink,capitalise]{cleveref}
\usepackage{xcolor}
\usepackage{soul} %
\usepackage{macros_journal}
\usepackage{caption}
\usepackage{subcaption}
\usepackage{textpos}

\usepackage{mathrsfs}
\setlist[description]{leftmargin=\parindent,labelindent=\parindent}
\makeatletter
\def\cleartheorem#1{%
    \expandafter\let\csname#1\endcsname\relax
    \expandafter\let\csname c@#1\endcsname\relax
}
\makeatother

\colorlet{myGreen}{green!50!black}
\colorlet{myLightgreen}{green}
\colorlet{myRed}{red!90!black}
\definecolor{myBlue}{rgb}{0.25, 0.0, 1.0}
\definecolor{myLightBlue}{rgb}{0.39, 0.58, 0.93}
\colorlet{myViolet}{myBlue!55!myRed}
\definecolor{myOrange}{rgb}{1.0, 0.66, 0.07}

\definecolor{magenta}{rgb}{0.94, 0.05, 0.53}
\definecolor{AO}{rgb}{0.0, 0.5, 0.0}
\definecolor{phthaloblue}{rgb}{0.0, 0.06, 0.54}
\definecolor{pistachio}{rgb}{0.58, 0.77, 0.45}
\definecolor{darkgoldenrod}{rgb}{0.72, 0.53, 0.04}

\hypersetup{
	colorlinks=true,
    linkcolor={red!50!black},
    citecolor={blue!50!black},
    urlcolor={blue!80!black},
	bookmarksopen=true,
	bookmarksnumbered,
	bookmarksopenlevel=2,
	bookmarksdepth=3
}

\newcommand{\AvN}[1]{\Fkt{Av}{G}} %
\newcommand{\AsN}[1]{\Fkt{As}{G}} %

 \title{Graphs with at most two moplexes}
\date{}

\preauthor{}
\DeclareRobustCommand{\authorthing}{
\begin{center}
Cl\'{e}ment Dallard\\
{\small LIFO EA 4022, INSA Centre Val de Loire, Université d'Orléans, Orléans, France} \\
\url{clement.dallard@univ-orleans.fr}\\
\smallskip
Robert Ganian\\ 
{\small Algorithms and Complexity Group, TU Wien, Vienna, Austria}\\
\url{rganian@gmail.com}\\ 
\smallskip
Meike Hatzel\\
{\small National Institute of Informatics, Tokyo, Japan}\\
\url{research@meikehatzel.com}\\
\smallskip
Matja\v{z} Krnc\\
{\small FAMNIT, University of Primorska, Koper, Slovenia} \\
\url{matjaz.krnc@upr.si}\\
\smallskip
Martin Milani{\v c}\\
{\small FAMNIT and IAM, University of Primorska, Koper, Slovenia} \\
\url{martin.milanic@upr.si}
\end{center}}
\author{\authorthing}
\postauthor{}
\begin{document}
\maketitle

\begin{textblock}{20}(-2, 3.2)
   \includegraphics[width=80px]{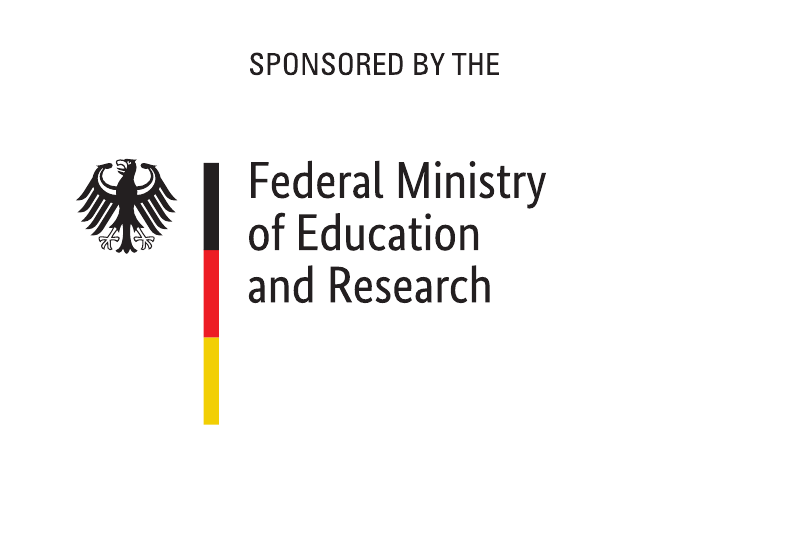}%
\end{textblock}

\begin{abstract}
    \blfootnote{This research was funded in part by the Slovenian Research Agency (I0-0035, research programs P1-0285 and P1-0383, and research projects J1-3001, J1-3002, J1-3003, J1-4008, J1-9110, N1-0102, N1-0160, and N1-0209).
  The authors gratefully acknowledge 
  the Republic of Slovenia (Investment funding of the Republic of Slovenia and the European Union of the European Regional Development Fund).
  Robert Ganian acknowledges support from the Austrian Science Fund (FWF project Y1329).
  Meike Hatzel's research has been supported by the European Research Council (ERC) under the European Union's Horizon 2020 research and innovation programme (ERC Consolidator Grant DISTRUCT, grant agreement No 648527) as well as by the Federal Ministry of Education and Research (BMBF) and by a fellowship within the IFI programme of the German Academic Exchange Service (DAAD).}

    A moplex is a natural graph structure that arises when lifting Dirac's classical theorem from chordal graphs to general graphs.
    The notion is known to be closely related to lexicographic searches in graphs as well as to asteroidal triples, and has been applied in several algorithms related to graph classes such as interval graphs, claw-free, and diamond-free graphs.
	However, while every non-complete graph has at least two moplexes, little is known about structural properties of graphs with a bounded number of moplexes.
	The study of these graphs is, in part, motivated by the parallel between moplexes in general graphs and simplicial modules in chordal graphs: unlike in the moplex setting, properties of chordal graphs with a bounded number of simplicial modules are well understood.
	For instance, chordal graphs having at most two simplicial modules are interval.
	
    In this work we initiate an investigation of $k$-moplex graphs, which are defined as graphs containing at most $k$ moplexes.
    Of particular interest is the smallest nontrivial case $k=2$, which forms a counterpart to the class of interval graphs. 
    As our main structural result, we show that, when restricted to connected graphs, the class of $2$-moplex graphs is sandwiched between the classes of proper interval graphs and cocomparability graphs; moreover, both inclusions are tight for hereditary classes.
    From a complexity theoretic viewpoint, this leads to the natural question of whether the presence of at most two moplexes guarantees a sufficient amount of structure to efficiently solve problems that are known to be intractable on cocomparability graphs, but not on proper interval graphs.
    We develop new reductions that answer this question negatively for two prominent problems fitting this profile, namely \textsc{Graph Isomorphism} and \textsc{Max-Cut}.
    On the other hand, we prove that every connected $2$-moplex graph contains a Hamiltonian path, generalising the same property of connected proper interval graphs.
    Furthermore, for graphs with a higher number of moplexes, we lift the previously known result that graphs without asteroidal triples have at most two moplexes to the more general setting of larger asteroidal sets.
    \blfootnote{A preliminary and shorter version of this article appeared in the proceedings of the XI Latin and American Algorithms, Graphs and Optimization Symposium (LAGOS 2021)~\cite{DALLARD2021-lagos}.}
\end{abstract}

\section{Introduction}
\label{sec:intro}

The fundamental class of \emph{chordal graphs}, i.e.\ graphs where every cycle of length at least four has a chord, has been extensively studied in the literature. 
A celebrated result by Dirac states that every non-complete chordal graph has at least two non-adjacent simplicial vertices~\cite{Dirac1961}.
Equivalently, every non-complete chordal graph contains at least two \emph{simplicial modules}, that is, maximal clique modules containing a simplicial vertex~\cite{Cao16,Shibatacliquetrees}.
Moreover, since simplicial modules are present in every chordal graph, one could classify chordal graphs into ``slices'' based on how many simplicial modules they have---for instance, it is easy to observe that every chordal graph with at most two simplicial modules is an interval graph (due to a well-known connection between simplicial modules and the \emph{leafage} of chordal graphs and a characterisation of interval graphs via leafage~\cite{lin1998leafage}).

On the other hand, there are many general graphs that do not contain any simplicial vertices, and attempting to use simplicial modules to slice up the class of all graphs in a similar manner would not be meaningful.
But while simplicial modules do not have the same fundamental role in general graphs as in chordal graphs, the notion of \emph{moplexes} introduced by Berry and Bordat~\cite{MR1626534} promises to hold precisely that role. 
In particular, a moplex is an inclusion-maximal set of vertices such that (1) they form a clique, (2) they form a module, and (3) their neighbourhood, if non-empty, is a minimal 
separator,\footnote{Precise definitions are provided in \cref{sec:prelims}.}
and Berry and Bordat lifted Dirac's theorem to general graphs by showing that every non-complete graph has at least two non-adjacent moplexes~\cite{MR1626534}. Subsequent works have pointed to
connections between moplexes and asteroidal triples~\cite{berry2001asteroidal,meister2005recognition} as well as lexicographic searches~\cite{BerryB01,corneil2004lexicographic,xu2013moplex,Trotignon16}.

Moplexes have also been used in various algorithms, e.g., for computing a minimal completion to an interval graph~\cite{RapaportST08}, for computing minimal triangulations of claw-free graphs~\cite{berry2012triangulation}, and for recognising diamond-free graphs without induced cycles of length at least five~\cite{berry2015efficiently}.

But in spite of these fundamental connections and useful applications, the interconnection between structural properties of graphs and their moplexes is still not well understood. 
In this work we initiate an investigation into what happens if one uses moplexes to slice up the class of general graphs. 
Do graphs containing a bounded number of moplexes have useful structural or algorithmic properties? And if chordal graphs with at most two simplicial modules form a natural subclass of the fundamental class of interval graphs, what can we say about graphs with at most two moplexes?

For a positive integer $k$, a \emph{$k$-moplex graph} is a graph that contains at most $k$ moplexes, and, moreover, we let the \emph{moplex number} of a graph be the number of moplexes it contains.
Our first, introductory, result provides a link between the moplex number and the \emph{asteroidal number} of a graph~\cite{MR1844876,CorneilOS97}, generalising an earlier result of Berry and Bordat for graphs with at most two~moplexes~\cite{berry2001asteroidal}.

\vspace{-\lastskip}\par \addvspace{.3pc}
\begin{restatable}{theorem}{numbers}\label{thm:an-mn}
The asteroidal number of a graph is a lower bound on its moplex number. 
\end{restatable}

\Cref{thm:an-mn} immediately implies that the nice algorithmic features of graphs with bounded asteroidal number also hold for graphs with a bounded number of moplexes.
This includes polynomial-time algorithms for various algorithmic problems~\cite{MR2019494,MR1686810,kratsch2012colouring,MR2432954},
existence of a spanning tree approximating vertex distances up to a constant additive term~\cite{MR1844876}, a constant factor approximation algorithm for treewidth~\cite{MR2159305}, and an upper bound on the treewidth in terms of the maximum degree~\cite{BODLAENDER199745}.
We remark that while computing the asteroidal number of a graph is \NP-hard~\cite{MR1639679}, the moplex number of a graph is polynomial-time computable~\cite{berry2001asteroidal}.

A graph class is \emph{hereditary} if it is closed under vertex deletion. 
The class of $1$-moplex graphs is hereditary, but not of particular interest, as it is precisely the class of complete graphs.
However, as one can verify using the family of paths, the class of $k$-moplex graphs is not hereditary for any $k\ge 2$.
This makes understanding the structure of $k$-moplex graphs significantly more challenging; in fact, even the structure of $2$-moplex graphs is not yet fully understood.
Berry and Bordat~\cite{berry2001asteroidal} showed that $2$-moplex graphs are AT-free and that all connected induced subgraphs of a graph $G$ are $2$-moplex if and only if $G$ is a proper interval graph.
We strengthen the former result and complement the latter by proving further results relating the class of $2$-moplex graphs to the hierarchy of hereditary graph classes.
More precisely, any graph class $\mathcal{G}$ can be naturally mapped to the following two hereditary graph classes, one contained in $\mathcal{G}$ and one containing~$\mathcal{G}$:
\begin{enumerate}
\item\label{item cochain} the class $\mathcal{G}^-$ of graphs all of whose induced subgraphs belong  to $\mathcal{G}$, or, equivalently, the largest hereditary graph class contained in $\mathcal{G}$, and
\item\label{item cocomp2} the class $\mathcal{G}^+$ of all induced subgraphs of graphs in $\mathcal{G}$, or, equivalently, the smallest hereditary graph class containing $\mathcal{G}$.
\end{enumerate}
Furthermore, if the graph class $\mathcal{G}$ is not closed under disjoint union (as is the case for the class of $2$-moplex graphs), it is also natural to consider the class $\mathcal{G}_c$ of all graphs $G$ such that every connected component of $G$ belongs to $\mathcal{G}$ and the corresponding two hereditary graph classes,  one contained in $\mathcal{G}_c$ and one containing~$\mathcal{G}_c$:
\begin{enumerate}[resume]
 \item \label{item PIGs} the class $\mathcal{G}_c^-$ of graphs all of whose induced subgraphs belong to $\mathcal{G}_c$, and
\item\label{item cocomp1} the class $\mathcal{G}_c^+$ of all induced subgraphs of graphs in $\mathcal{G}_c$.
\end{enumerate}
As previously mentioned, when $\mathcal{G}$ is the class of $2$-moplex graphs, a result of Berry and Bordat shows that the corresponding class from \hyperlink{item PIGs}{(iii)} is the class of proper interval graphs. 
We determine the remaining three hereditary classes related to $2$-moplex graphs.
We show that the corresponding class from~\hyperlink{item cochain}{(i)} is the class of cochain graphs, while the classes from~\hyperlink{item cocomp1}{(ii)} and~\hyperlink{item cocomp2}{(iv)} both coincide with the class of cocomparability graphs.
In particular, we prove the following two results.
\begin{restatable}{theorem}{CochainsMaximalHereditarySubclass}\label{thm:cochain}
Let $G$ be a graph. 
Then, every induced subgraph of $G$ is a $2$-moplex graph if and only if $G$ is a cochain graph.
\end{restatable}

\vspace{-\lastskip}\par \addvspace{.3pc}
\begin{restatable}{theorem}{CocompIsMinimalHereditarySuperclass}\label{cocomp is minimal superclass of M2}
The smallest hereditary graph class containing the class of $2$-moplex graphs is the class of cocomparability graphs.
\end{restatable}
\par\addvspace{.6pc}

We then consider the question of whether the structure of $2$-moplex graphs can be used to develop efficient algorithms for problems that are known to be intractable on cocomparability graphs, but not on proper interval graphs. 
We develop reductions showing that this is not the case for two prominent examples of such problems, namely \textsc{Max-Cut} and \textsc{Graph Isomorphism}, both of which remain as hard on $2$-moplex graphs as they are on cocomparability graphs.
For proper interval graphs, the complexity of \textsc{Max-Cut} is still open, while \textsc{Graph Isomorphism} is solvable in linear time~\cite{MR528025}.%

\begin{restatable}{theorem}{MaxCutNPComplete}\label{thmMaxCutNPComplete}
        \textsc{Max-Cut} is \NP-complete on cobipartite $2$-moplex graphs.%
\end{restatable}%

\begin{restatable}{theorem}{GraphIsomorphismGIComplete}\label{thmGraphIsomorphismGIComplete}
    \textsc{Graph Isomorphism} is \GI-complete on cobipartite $2$-moplex graphs.%
\end{restatable}%

\Cref{thmMaxCutNPComplete,thmGraphIsomorphismGIComplete} provide some indication that the class of $2$-moplex graphs is a significant generalisation of the class of connected proper interval graphs.
Nevertheless, as our final result we show that $2$-moplex graphs share the well-known structural property of proper interval graphs that connectedness is 
a sufficient condition for the existence of a Hamiltonian path~\cite{MR731128}.

\begin{restatable}{theorem}{MtwoHamiltonicity}\label{thmMtwoHamiltonicity}
    Every connected $2$-moplex graph has a Hamiltonian path.
\end{restatable}%

\Cref{thmMtwoHamiltonicity} resolves a conjecture from the preliminary version of this paper~\cite{DALLARD2021-lagos}.

The proof of this theorem is an interplay between properties of the class of cocomparability graphs, the Lexicographic Depth First Search algorithm, and the concept of avoidable vertices (also known as OCF-vertices)~\cite{MR4357319,bonamy2020avoidable,ohtsuki1976minimal}.

\section{Preliminaries}
\label{sec:prelims}

All graphs in this paper are finite, undirected, and simple. We assume familiarity with basic concepts in graph theory as used, e.g., by West~\cite{MR1367739}. 
We denote the vertex set and the edge set of a graph $G$ by $\V{G}$ and $\E{G}$, respectively. 
The closed and open neighbourhoods of a vertex $v$ in $G$ are denoted as $N[v]$ and $N(v)$, respectively.
These concepts are naturally extended to sets $X\subseteq V(G)$ so that $N[X]$ is defined as the union of all closed neighbourhoods of vertices in $X$, and $N(X)$ is defined as the set $N[X]\setminus X$.
A~\emph{clique} in $G$ is a set of pairwise adjacent vertices.
An~\emph{independent set} is a set of pairwise non-adjacent vertices. 
The \emph{independence number} of $G$ is the maximum size of an independent set in $G$.
A~\emph{(connected) component} of a graph is a maximal connected subgraph.
We sometimes identify components of a graph with their vertex sets.
Given two vertex sets $A$ and $B$ in $G$, we say that $A$ \emph{dominates} $B$ if every vertex in $B$ has a neighbour in $A$.

A \emph{cut-vertex} in a graph is a vertex whose deletion increases the number of connected components.
For a graph $G$ and $S\subseteq V(G)$, let $G-S$ be the subgraph of $G$ induced by $V-S$; if $S=\{v\}$ then we use $G-v$ as shorthand for $G-\{v\}$.
For two vertices $u,v \in \V{G}$, a set $S\subseteq V(G)\setminus\{u,v\}$ is a \emph{$u{,}v$-separator} if $u$ and $v$ belong to different components of $G-S$, and it is a \emph{minimal $u{,}v$-separator} if additionally no proper subset of $S$ is a $u{,}v$-separator.
A~$u{,}v$-separator $S$ is minimal if and only if the two components of $G-S$ containing $u$ and $v$ are \emph{$S$-full}, that is, these components dominate $S$.
A~\emph{minimal separator} is a minimal $u{,}v$-separator for two non-adjacent vertices $u$ and~$v$.

Given two graphs $G$ and $H$, we say that $G$ is \emph{$H$-free} if no induced subgraph of $G$ is isomorphic to $H$.
A \emph{star} is a connected graph where all edges are incident to one vertex, and the \emph{claw} is the star on $4$ vertices.
We denote by $3K_1$ the edgeless graph with three vertices, and, for $k\geq 3$, by $C_k$ the cycle of length $k$.
A graph is \emph{chordal} if it does not contain any induced cycle of length greater than $3$.

Besides the class of chordal graphs, several other hereditary graph classes play an important role in our study.
A graph $G$ is \emph{cobipartite} if its complement is bipartite, that is, if the vertex set of $G$ can be partitioned into two cliques.
Furthermore, $G$ is said to be a \emph{cochain graph} if we can write $V(G) = X\cup Y$ where $X$ and $Y$ are disjoint cliques with $X= \{x_1,\ldots, x_k\}$ such that $N[x_i]\subseteq N[x_j]$ for all $1\le i<j\le k$.
A graph $G$ is an \emph{interval graph} if it has an \emph{interval representation}, that is, if its vertices can be put in a one-to-one correspondence with a family of closed intervals on the real line such that two distinct vertices are adjacent if and only if the corresponding intervals intersect.
If $G$ has an interval representation in which no interval contains another interval, then $G$ is said to be a \emph{proper interval graph}.
A vertex set $A$ in a graph $G$ is an \emph{asteroidal set} if for each $a\in A$, the vertices in $A\setminus \{a\}$ are all contained in a single connected component of $G-N[a]$~\cite{MR505894}. 
Asteroidal sets of cardinality three are called \emph{asteroidal triples}, and graphs not containing any asteroidal triples are called \emph{AT-free}. 
A prominent subclass of AT-free graphs is the class of \emph{cocomparability graphs}, which are graphs whose complements allow for a transitive orientation of their edges.

Roberts proved that proper interval graphs are exactly the claw-free interval graphs~\cite{MR0252267} (see also~\cite{MR2364171,MR1687858}).
This result, together with the characterisation of interval graphs due to Lekkerkerker and Boland stating that interval graphs are exactly the AT-free chordal graphs~\cite{MR139159}, implies the following result.

\begin{theorem}\label{thm:PIGs}
A graph $G$ is a proper interval graph if and only if $G$ is a claw-free AT-free chordal graph.
\end{theorem}

Let us introduce some notation for orderings of the vertex set of a graph.
A \emph{(vertex) ordering} of a graph $G$ is a total order on its vertex set. Given an ordering $\sigma$ of a graph $G$ and two distinct vertices $x$ and $y$, we write $x <_{\sigma} y$ if $x$ precedes $y$ in $\sigma$.
Three vertices $x <_{\sigma} y <_{\sigma} z$ of $G$ with $xy \notin \E{G}$, $yz \notin \E{G}$, and $xz \in \E{G}$ are said to form an \emph{umbrella} in $\sigma$.
An ordering is \emph{umbrella-free} if no three vertices form an umbrella in it.
As shown by Kratsch and Stewart~\cite{kratsch1993domination}, the existence of an umbrella-free ordering characterises cocomparability graphs.

\begin{theorem}[Kratsch and Stewart~\cite{kratsch1993domination}]\label{thm:umbrella}
A graph $G$ is a cocomparability graph if and only if it has an umbrella-free ordering.
\end{theorem}

For further background on graph classes, we refer to~\cite{MR1686154}. 

A vertex is \emph{simplicial} if its neighbourhood forms a clique.
A vertex set $M$ is a \emph{module} if each vertex $v\in V(G)\setminus M$ is either adjacent to every vertex in $M$ or not adjacent to any vertex in $M$. A \emph{clique module} is a module that is a clique.
A \emph{simplicial module} is an inclusion-maximal clique module containing a simplicial vertex.
Note that all vertices in a simplicial module are simplicial.

Let $v$ be a vertex in a graph $G$.
An \emph{extension} of $v$ in $G$ is an induced $P_3$ in $G$ having $v$ as midpoint.
A vertex $v$ is \emph{avoidable} in $G$ if every extension of $v$ is contained in an induced cycle.
Note that every simplicial vertex is avoidable, and a vertex in a chordal graph is avoidable if and only if it is simplicial.
The concept of avoidable vertices goes back to the work of Ohtsuki, Cheung, and Fujisawa~\cite{ohtsuki1976minimal}, who proved that every graph has an avoidable vertex.
Vertices having this property were later dubbed \emph{OCF-vertices}~\cite{MR2057267,MR2204109,MR2191642,yang2014ordering,BerryBBS10}, and very recently the notion has reemerged under the name of avoidable vertices~\cite{MR4357319,bonamy2020avoidable}.

We conclude our preliminaries with a formal definition of \emph{moplexes}, another strengthening of avoidable vertices, which play a central role in this paper. 
For purely technical reasons we use the definition from~\cite{meister2005recognition}, which extends the one in~\cite{MR1626534} so that the vertex set of any complete graph is also a moplex (see also~\cite{MR2816655,xu2013moplex}).
An illustration of the notions of moplexes and avoidable vertices is provided in \cref{fig:moplex}.

\begin{definition}%
A \emph{moplex} in a graph $G$ is an inclusion-maximal clique module $X \subseteq \V{G}$ 
such that $N(X)$ is either empty or
a minimal separator in $G$.
A moplex $X$ is \emph{simplicial} if $N(X)$ is a clique.
A vertex is \emph{moplicial} if it belongs to a moplex.
\end{definition}

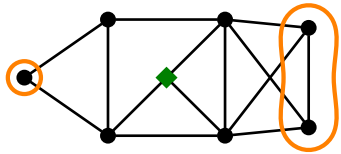
\begin{figure}[thb]
        \centering
        \begin{tikzpicture}[scale=1.1]
        \tikzstyle{vertex} = [draw, circle, scale=0.5, thick,fill]
        \node[vertex,green!50!black,diamond] (Center) at (0,0) {};
        \node[vertex] (A) at ($(Center)+(-0.7,0.7)$) {};
        \node[vertex] (B) at ($(Center)+(0.7,0.7)$) {};
        \node[vertex] (C) at ($(Center)+(0.7,-0.7)$) {};
        \node[vertex] (D) at ($(Center)+(-0.7,-0.7)$) {};
        
        \node[vertex] (B2) at ($(B)+(1,-0.1)$) {};
        \node[vertex] (C2) at ($(C)+(1,0.1)$) {};
        
        \node[vertex] (A2) at ($(Center)+(-1.7,0)$) {};
        
        \draw[line width=1pt] (A) -- (B);
        \draw[line width=1pt] (B) -- (C);
        \draw[line width=1pt] (C) -- (D);
        \draw[line width=1pt] (D) -- (A);
        \draw[line width=1pt] (Center) -- (B);
        \draw[line width=1pt] (Center) -- (C);
        \draw[line width=1pt] (Center) -- (D);
        \draw[line width=1pt] (A) -- (A2);
        \draw[line width=1pt] (A2) -- (D);
        
        \draw[line width=1pt] (B) -- (B2);
        \draw[line width=1pt] (B) -- (C2);
        \draw[line width=1pt] (C) -- (B2);
        \draw[line width=1pt] (C) -- (C2);
        \draw[line width=1pt] (B2) -- (C2);
        
        \node[draw, circle, scale=1.2,ultra thick,orange] at ($(A2)$) {};
        \draw[ultra thick,orange] ($(B2)+(-0.2,0.2)$) to[closed=true,curve through={($(Center)+(1.4,0)$) .. ($(C2)+(-0.2,-0.2)$) .. ($(C2)+(0.2,-0.2)$) .. ($(Center)+(2,0)$)}] ($(B2)+(0.2,0.2)$);
    \end{tikzpicture}
    \caption{A graph with exactly two moplexes (circled in orange) containing an avoidable vertex (the green diamond) that is neither moplicial nor simplicial.}
    \label{fig:moplex}
\end{figure}

Note that if $X$ is a moplex in a graph $G$, then the graph $G-N[X]$ contains an $N(X)$-full component.
Using such a component, it can be shown that every extension of a vertex $v\in X$ is contained in an induced cycle in $G$.
This leads to the following observation (as noted already in~\cite{berry2005extremities}, and perhaps earlier).

\begin{observation}\label{moplicial implies avoidable}
Every moplicial vertex in a graph is avoidable.
\end{observation}

We use this observation frequently in the paper (without explicit reference).

\section{Structural properties}\label{sec:inclusions}

In this section, we focus on establishing structural properties of $k$-moplex graphs for a fixed $k$, with a special focus on the smallest nontrivial graph class defined in this way, namely the class of $2$-moplex graphs.
We begin by recalling a result on minimal separators and moplexes.

\begin{theorem}[Berry and Bordat~\cite{berry2001asteroidal}]\label{moplex in each CC}
For every minimal separator $S$ in a graph $G$, 
each component of $G-S$ contains a moplex in~$G$.
\end{theorem}

In fact, the same authors previously showed that this result holds under the assumption that the graph $G$ is chordal.

\begin{theorem}[Berry and Bordat~\cite{MR1626534}]\label{chordal moplex in each cc}
    Let $S$ be a minimal separator in a chordal graph $H$.
    Then each connected component of $H-S$ contains at least one moplex.
\end{theorem}

\Cref{moplex in each CC} is stated in~\cite{berry2001asteroidal} without a proof. 
For the sake of completeness, we give a short proof here. To do so, we need two more results from the literature. 
We call a chordal graph $G'$ a \emph{minimal triangulation} of a graph $G$ if $\V{G}=\V{G'}$, $\E{G}\subseteq \E{G'}$, and for all $F\subsetneq E(G')\setminus E(G)$, the graph $(\V{G},\E{G}\setminus F)$ is not chordal.

\begin{lemma}[Berry and Bordat~\cite{MR1626534}]\label{moplex in H is moplex in G}
Let $H$ be a minimal triangulation of a graph $G$ and $U$ be a moplex of $H$.
Then $U$ is a moplex of $G$.
\end{lemma}

The next theorem is an immediate consequence of \cite[Theorem 4.6]{MR1478250}.

\begin{theorem}\label{min sep chordal}
    Let $S$ be a minimal separator in a graph $G$.
    Then there exists a minimal triangulation $H$ of $G$ such that
    the vertex sets of the connected components of $H -S$ are the same as the vertex sets of the connected components of $G - S$.
\end{theorem}

We are now ready to prove \cref{moplex in each CC}.

\renewcommand*{\proofname}{Proof of \cref{moplex in each CC}.}
\begin{proof}
    Let $S$ be a minimal separator in a graph $G$.
    According to \cref{min sep chordal}, there exists a minimal triangulation $H$ of $G$ such that, for each connected component $C_i$ in $G-S$ there exists a connected component $C_i'$ in $H$ with $V(C_i) = V(C_i')$.
    This readily implies that $S$ is a minimal separator in $H$.
    Then, \cref{chordal moplex in each cc} guarantees that there exists a moplex in each connected component of $H - S$ and \cref{moplex in H is moplex in G} implies that these moplexes are moplexes in $G$ as well.
\end{proof}
\renewcommand*{\proofname}{Proof}

With \cref{moplex in each CC} at hand, we can already obtain some preliminary results on graphs with at most two moplexes.
These will be useful in \cref{sub:conwithcocomp}.

\begin{lemma}\label{M2 moplexes properties}
    Let $G$ be a non-complete $2$-moplex graph and denote by $U$ and $W$ its two moplexes.
    Then the following holds:
    \begin{enumerate}
        \item\label[property]{U and W are disjoint simplicial moplexes} $U$ and $W$ are disjoint simplicial moplexes, and
        \item\label[property]{G-S two cc} for every minimal separator $S$ in $G$, the graph $G -S$ contains exactly two connected components, one of which contains $U$ and the other one $W$.
    \end{enumerate}
\end{lemma}

\begin{proof}
    First, we show that \cref{U and W are disjoint simplicial moplexes} holds.
    The fact that $U$ and $W$ are disjoint  follows directly from the definition of a moplex, and in particular because every moplex is a maximal set of vertices with the same closed neighbourhood.
    Suppose without loss of generality that $U$ is not a simplicial moplex, that is, there exist two non-adjacent vertices $a,b \in N(U)$.
    Let $S$ be a minimal $a{,}b$-separator.
    Then $S$ contains $U$, as every vertex in $U$ is adjacent to both $a$ and $b$.
    However, following \cref{moplex in each CC}, the graph $G-S$ contains at least two moplexes in $G$, say $M_1$ and $M_2$. But then $M_1$, $M_2$, and $U$ are three distinct moplexes in $G$, a contradiction to $G$ being a $2$-moplex graph.
    
    Second, we show \cref{G-S two cc}.
    Let $S$ be a minimal separator in $G$.
    If there exists a connected component of $G-S$ that does not contain $U$ or $W$, then \cref{moplex in each CC} implies that $G$ contains at least three moplexes, a contradiction.
    Thus, $G - S$ must contain exactly two connected components, 
    one of which contains $U$ and the other one $W$.
\end{proof}

\subsection{The asteroidal number is a lower bound on the moplex number}

We generalise a result by Berry and Bordat stating that a graph has an asteroidal triple of vertices if and only if it has an asteroidal triple of moplexes~\cite{berry2001asteroidal}.
The \emph{asteroidal number} of a graph $G$ is defined as the maximum size of an asteroidal set (see, e.g.,~\cite{lin1998leafage,MR1844876,MR3159134}).
An \emph{asteroidal set of moplexes} in a graph $G$ is a set $\{X_1,\ldots, X_k\}$ of pairwise disjoint moplexes in $G$ such that  for each $i\in \{1,\ldots, k\}$, the graph $G-N[X_i]$ contains all 
moplexes $X_j$, $j\neq i$, in the same connected component.
The result of Berry and Bordat~\cite{berry2001asteroidal} corresponds to the case $k = 3$ of the following more general statement.

 \begin{theorem}\label{thm:asteroidal}
 A graph has an asteroidal set of vertices of size $k$ if and only if it has an asteroidal set of moplexes of size $k$.
 \end{theorem}

\begin{proof}
Let $G$ be a graph and let $k$ be a positive integer.
If $G$ has an asteroidal set of moplexes of size $k$, then $G$ has an asteroidal set of vertices of size $k$.

Let $A = \{a_1, \dots, a_k\}$ be an asteroidal set of size $k$ in $G$ such that the number of moplicial vertices in $A$ is as large as possible.
First note that since $A$ is an independent set and every moplex is a clique, no two vertices in $A$ can belong to the same moplex.
Thus, to complete the proof it suffices to show that every vertex $a_i\in A$ is moplicial.
Indeed, denoting by $M_i$ the moplex of $G$ containing $a_i$, for all $i\in \{1,\ldots, k\}$, we would obtain that $\{M_1,\ldots, M_k\}$ is an asteroidal set of moplexes of size $k$.
Suppose that some $a_i \in A$ is not part of a moplex.
Let $C$ be the component of $G - N[a_i]$ containing $A \setminus \{a_i\}$ and let $D$ be the component of $G - N[V(C)]$ containing $a_i$.
Let $S = N(V(C))$ and observe that $S \subseteq N(a_i)$.
This implies that both $C$ and $D$ are $S$-full components of $G-S$, and thus that $S$ is a minimal $a_i{,}a_j$-separator for any $a_j \in A \setminus \{a_i\}$.
Hence, we can use \cref{moplex in each CC} and obtain that $D$ contains a moplex $M_D$ in $G$.
Let $x$ be a vertex in $M_D$ and $A' = (A \setminus \{a_i\}) \cup \{x\}$.

We show next that $A'$ is an asteroidal set in $G$.
Fix two vertices $v,w \in A' \setminus \{x\}$.
Note that $N[x] \subseteq D \cup S$, and since both $v$ and $w$ belong to $C$, there exists a $v{,}w$-path that does not contain any vertex from $N[x]$.
Furthermore, observe that there exists a path $P$ between $a_i$ and $x$ in $D$, which does not contain any vertex in $N[w]$.
Also, since $A$ is an asteroidal set, there exists a path $P'$ between $a_i$ and $v$ which does not contain any vertex in $N[w]$.
Hence, the subgraph $P\cup P'$ contains a walk between $x$ and $v$ in $G - N[w]$.
Therefore, $A'$ is an asteroidal set in $G$ of size $k$.
However, the number of vertices in $A'$ that belong to a moplex is strictly larger than the number of vertices in $A$ that belong to a moplex, contradicting the choice of $A$.
This shows that every vertex in $A$ is moplicial.
\end{proof}

\Cref{thm:an-mn} is an immediate consequence of \cref{thm:asteroidal}.

\numbers*

On the other hand, the gap between the moplex and asteroidal number can be arbitrarily large (consider, e.g., the class of stars). 

\Cref{thm:an-mn} implies that the asteroidal number is computable in polynomial time in any class of graphs with bounded moplex number.
Together with results from~\cite{MR2019494}, \cref{thm:an-mn} also implies that \textsc{Dominating Set} and \textsc{Total Dominating Set} can be solved in polynomial time in classes of graphs of bounded moplex number.
The same holds for \textsc{Independent Set},
\textsc{Independent Dominating Set},
and 
\textsc{Efficient Dominating Set},
along with their weighted variants~\cite{MR1686810}, for \textsc{$k$-Colouring},
for any fixed $k$~\cite{kratsch2012colouring}, and for \textsc{Weighted Feedback Vertex Set}~\cite{MR2432954}.
Furthermore, for graphs of bounded moplex number, bounded maximum degree implies bounded treewidth, as a consequence of the fact that graphs with asteroidal number at most $k$ have chordality at most $2k+1$ and of a result by Bodlaender and Thilikos showing that bounded chordality and bounded maximum degree implies bounded treewidth~\cite{BODLAENDER199745}.

For later use, we explicitly state the previous result of Berry and Bordat on $2$-moplex graphs~\cite{berry2001asteroidal} (which is now an immediate consequence of \cref{thm:an-mn}).

\begin{corollary}\label{graphs in M2 are AT-free}
Every $2$-moplex graph is AT-free.
\end{corollary}

\subsection{On vertex deletion in \texorpdfstring{$k$}{k}-moplex graphs}\label{sec:Removing_avoidable_vertices}

Since the class of $k$-moplex graphs is not hereditary for any $k \geq 2$, it is a natural question to ask whether for all $k$, every $k$-moplex graph contains a vertex whose removal results in a $k$-moplex graph. 
Unfortunately this is not the case. 
It is not difficult to show that, for every $k\ge 2$, there exists a $k$-moplex graph $G_k$ such that deleting any vertex results in a graph that is not a $k$-moplex graph. 
See \cref{fig:counterexamples} for examples of such graphs for $k\in \{2,3,4\}$; the construction can be easily generalised to larger values of $k$. 

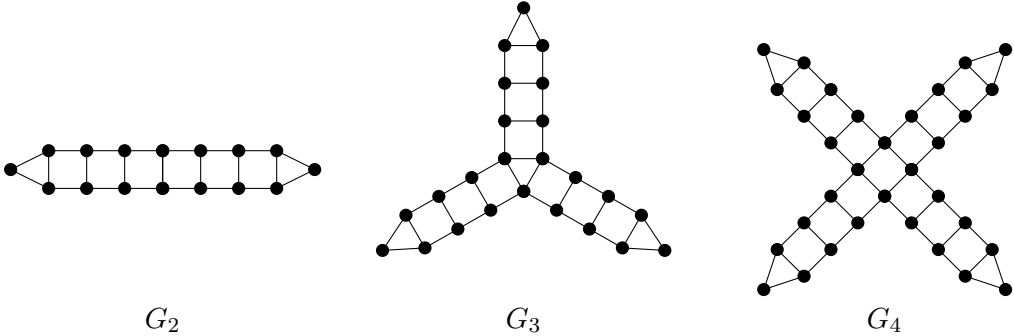
\begin{figure}[thb]
	\centering
    \begin{tikzpicture}[scale=0.5]
\tikzset{vertex/.style = {draw, circle, scale=0.4, thick, fill}}
\tikzset{
  pics/multihouse/.style n args={5}{
    code = {
          \coordinate (center) at ($ (#1,#2) + (#4:#5) $);
          \node[vertex] (0left) at ($ (center) + (90+#4:0.5*\scale)$) {};
          \node[vertex] (0right) at ($ (center) + (-90+#4:0.5*\scale) $) {};
          \draw (0left) to (0right);
          \foreach \y[evaluate={\z=int(\y -1)}] in {1,...,#3}{
            \node[vertex] (\y left) at ($ (0left) + (#4:\y*\scale) $) {};
            \node[vertex] (\y right) at ($ (0right) + (#4:\y*\scale) $) {};
            \draw (\y left) to (\y right)
                  (\z left) to (\y left)
                  (\z right) to (\y right);
          }
          \node[vertex] (m) at ($ (center) + (#4:{(#3+1)*\scale}) $) {};
          \draw (#3left) to (m) to (#3right);
    }
  }
}

\def\scale{0.5}
\begin{scope}[xshift=-9.5cm]
  \def\k{2}
  \def\r{0};
  \foreach \a in {1,...,\k}{
      \draw pic {multihouse={0}{0}{3}{\a*360/\k}{\r}};
  }
    \node at (0,-4) {$\mathstrut G_2$};
\end{scope}

\begin{scope}[]
  \def\k{3}
  \def\r{{\scale/(2*tan(180/\k))}};
  \foreach \a in {1,...,\k}{
      \draw pic {multihouse={0}{0}{3}{-30+\a*360/\k}{\r}};
  }
    \node at (0,-4) {$\mathstrut G_3$};
\end{scope}

\begin{scope}[xshift=9.5cm]
  \def\k{4}
  \def\r{{\scale/(2*tan(180/\k))}};
  \foreach \a in {1,...,\k}{
      \draw pic {multihouse={0}{0}{3}{-45+\a*360/\k}{\r}};
  }
  \node at (0,-4) {$\mathstrut G_4$};
\end{scope}
\end{tikzpicture}
	\caption{Examples of graphs in which deleting any vertex increases the moplex number. 
    In all these graphs, at least one neighbour of the removed vertex becomes moplicial.
    Moreover, existing moplicial vertices either remain moplicial or, if deleted, lead to two neighbours becoming moplicial.
 }\label{fig:counterexamples}
\end{figure}

Nevertheless, we show that there are still certain vertices that, if present, can be removed without leaving the class of $k$-moplex graphs. To this end, we first prove the following stronger statement.
 
 \begin{theorem}\label{thm:removing_avoidable_vertices}
    Let $G$ be a graph, $v \in V(G)$ an avoidable vertex that is not moplicial, and $M\subseteq V(G)$.
    Then $M$ is a moplex in $G$ if and only if it is a moplex in $G-v$.
\end{theorem}

 \begin{proof}
    Observe that if there is a vertex $v' \in V(G) \setminus \{v\}$ such that $N[v'] = N[v]$, then there is a natural correspondence between the clique modules of $G$ and those of $G-v$, which preserves, in both directions, the property of being a moplex.\footnote{Every neighbourhood notation of the form $N(\cdot)$, $N[\cdot]$ in this proof is considered to be in $G$.} Thus, from now on, we assume that for any vertex $v' \neq v$ in $G$, it holds that $N[v'] \neq N[v]$.
     
    Suppose that $M$ is a moplex in $G$ but not a moplex in $G-v$, and let $S=N(M)$. 
    By assumption, $v$ is not moplicial in $G$, and thus $v \notin M$.
    As $M$ is a clique and a module in $G-v$, there is no $S$-full component of $G-(S\cup \{v\})$ other than $M$ itself, otherwise $M$ would remain a moplex in $G-v$.
    Since such a component exists in $G-S$, call it $H$, we infer that $v$ belongs to this component and that there is a vertex $x \in S$ such that $v$ is the only neighbour of $x$ in $H$, that is, $N(x) \cap V(H) = \{v\}$.
    Let $H'$ be the connected component of $G-N[v]$ containing $M$.
    We claim that $H'$ is $N(v)$-full.
    We know that $N[v] \subseteq S \cup \V{H}$.
    Since $M$ lies in $H'$, every vertex in $N(v) \cap S$ has a neighbour in $H'$.
    Consider now a vertex $w \in N(v) \setminus S$.
    Then $w\in V(H)$ and, since $v$ is the only vertex in $H$ that is adjacent to $x$, we infer that $w$ is not adjacent to $x$.
    Thus, we get an extension $wvx$ and, since $v$ is avoidable in $G$, a $w{,}x$-path $P$ in $G-(N[v]\setminus \{w,x\})$.
    Let $y$ be the first vertex of $P$ not in $H$.
    Then $y$ belongs to $S$ and consequently to $H'$.
    Furthermore, since $v$ is the only neighbour of $x$ in $H$, we have $y \neq x$.
    Hence, the path $P'$ obtained from $P$ by removing $x$ and $w$ lies in $H'$, and so $w$ has a neighbour in $H'$.
    We obtain that every neighbour of $v$ has a neighbour in $H'$, and thus that $H'$ is $N(v)$-full. 
    It follows that $\{v\}$ is a moplex in $G$, a contradiction.
    
    For the backward direction, suppose towards a contradiction that there exists a moplex $M$ in $G-v$ which is not a moplex in $G$.
    Let $S$ be the neighbourhood of $M$ in $G-v$.
    Since $S$ is a minimal separator in $G-v$, there exists an $S$-full component $H$ in $(G-v)-N[M]$.
    Assume first that $v$ has no neighbours in $M$.
    Then $N(M) = S$, and hence the component of $G-S$ containing the vertices of $H$ contains no vertices of $M$ and is $S$-full. 
    This implies that $M$ is a moplex in $G$, a contradiction.
    Hence, there exists a vertex $u \in N(v) \cap M$.
    
    We claim that $M \subseteq N(v)$.
    Towards a contradiction, suppose there exists a vertex $u' \in M \setminus N(v)$.
    Let $H'$ be the component of $G-N[v]$ containing $u'$.
    We show that $H'$ is $N(v)$-full in $G$.
    Since $u'\in M$, all the vertices in $(M\cup S)\setminus \{u'\}$ are adjacent to $u'$. 
    Thus, as $u'$ is in $H'$, every vertex in $N(v)\cap (M \cup S)$ has a neighbour in $H'$.
    So, let $z \in N(v) \setminus (M \cup S)$.
    Note that $u\in M$ and $z\not\in S\cup \{v\} = N(M)$, and hence the vertices $u$ and $z$ are non-adjacent in $G$, which implies that $uvz$ is an extension of $v$.
    Since $v$ is avoidable, there is a $u{,}z$-path in $G-(N[v]\setminus \{u,z\})$.
    Any such path intersects $S\setminus N(v)$, and thus all its internal vertices are in the same component of $G-N[v]$ as $u'$, namely $H'$. In particular, this is the case for the neighbour of $z$ on the path.
    This shows that every vertex in $N(v) \setminus (M \cup S)$ has a neighbour in $H'$.
    We conclude that $H'$ is $N(v)$-full in $G$, as claimed.
    Thus, $\{v\}$ is a moplex in $G$, a contradiction.

    From the previous observation that $M \subseteq N(v)$, we readily get that $v$ does not have any neighbours in $H$.
    Indeed, if $v$ has a neighbour in $H$, then $M$ would be a clique module in $G$ with neighbourhood $S \cup \{v\}$ and $H$ would be an $(S \cup \{v\})$-full component, and hence $M$ would be a moplex in $G$, a contradiction.
        
    We show next that $N(v) \setminus (M \cup S)$ is non-empty.
    Suppose that $N(v) \subseteq M \cup S$.
    Recall that $N[v] \neq N[u]$.
    Thus, since $N[u] = M \cup S \cup \{v\}$, there exists a vertex $x \in N[u]\setminus N[v]$.
    Then $x\in M\cup S$; however, as $M \subseteq N(v)$, we must have $x \in S$.
    Following the fact that $S = N(V(H))$ and $M\subseteq N(x)$, we obtain that $(S\cup M)\setminus\{x\}\subseteq N(V(H) \cup \{x\})$, and hence $N(v) \subseteq N(V(H) \cup \{x\})$.
    As previously shown, $v$ has no neighbours in $H$.
    Hence, the vertices in $V(H) \cup \{x\}$ all belong to the same connected component of $G-N(v)$.
    However, this implies that $N(v)$ is a minimal separator in $G$, and thus $\{v\}$ is a moplex in $G$, a contradiction.
    Thus, $N(v) \setminus (M \cup S)$ is non-empty.
    
    Let $z \in N(v) \setminus (M \cup S)$. 
    Then $u$ and $z$ are non-adjacent in $G$ and $uvz$ is an extension of $v$.
    Since $v$ is avoidable, there exists an induced $u{,}z$-path in $G$ having all internal vertices in $G-N[v]$.
    In particular, this path must contain a vertex in $S \setminus N(v)$, and thus $S\setminus N(v) \neq \emptyset$.
    Following the fact that $N(v) \cap V(H) = \emptyset$, there exists a component $H'$ of $G-N[v]$ that contains all vertices of $H$.
    We show that $H'$ is an $N(v)$-full component in $G$.
    As every vertex of $S$ has a neighbour in $H$, we infer that every vertex in $S\setminus N(v)$ belongs to $H'$, while every vertex in $N(v)\cap S$ has a neighbour in $H'$.
    Furthermore, since $\emptyset \neq S\setminus N(v) \subseteq \V{H'}$, every vertex of $N(v)\cap M = M$ has a neighbour in $H'$.
    Finally, as $v$ is avoidable, for every vertex $y \in N(v)\setminus(M \cup S)$ there is a $u{,}y$-path in $G-(N[v]\setminus\Set{u,y})$.
    Any such path intersects $S \setminus N(v)$, and thus all its internal vertices are contained in $H'$.
    Hence, $\Set{v}$ is a moplex in $G$, a contradiction.
 \end{proof}

We can now establish the announced claim, which is later applied in \cref{sec:traceable}.

\begin{corollary}\label{remove_avoidable_vertices}
For every positive integer $k$, the class of $k$-moplex graphs is closed under deletion of an avoidable vertex that is not moplicial.
\end{corollary}

\begin{proof}
Let $G$ be a $k$-moplex graph, and let $v \in V(G)$ be an avoidable vertex that is not moplicial.
By \cref{thm:removing_avoidable_vertices}, every moplex in $G-v$ is a moplex in $G$. 
Thus, since $G$ has at most $k$ moplexes, so does $G-v$.
\end{proof}

\subsection{Connections with proper interval graphs and cochain graphs}
\label{sec:PIGs}

Berry and Bordat characterised graphs in which every connected induced subgraph has at most two moplexes by the well-known class of proper interval graphs.
Since \cref{thm:hereditary-connected-M2} was stated in~\cite{berry2001asteroidal} without a detailed proof, we offer a short proof.

\begin{theorem}[Berry and Bordat~\cite{berry2001asteroidal}]\label{thm:hereditary-connected-M2}
Let $G$ be a graph. Then, each connected induced subgraph of $G$ has at most two moplexes if and only if $G$ is a proper interval graph.
\end{theorem}

\begin{proof}
Suppose first that each connected induced subgraph of $G$ has at most two moplexes. Note that for every $k\ge 4$, the cycle $C_k$ is a connected graph in which every vertex forms a moplex. Furthermore, the claw is a connected graph in which every vertex of degree one forms a moplex. Thus, $G$ must be claw-free and chordal. 
By \cref{graphs in M2 are AT-free}, this implies that $G$ is a claw-free AT-free chordal graph.
Using \cref{thm:PIGs} we conclude that $G$ is a proper interval graph.

For the converse direction, let $G$ be a proper interval graph and let $H$ be a connected induced subgraph of $G$. Then, $H$ is a connected proper interval graph.
We need to show that $H$ has at most two moplexes.
Since both the connectedness and the moplex number are preserved upon deleting a vertex from a pair of vertices with the same closed neighbourhoods, we may assume that no two vertices in $H$ have the same closed neighbourhood.
Under this assumption, every moplex in $H$ has size one.
To complete the proof, fix a proper interval model of $H$ and let $I_1, \ldots, I_n$ be the ordering of the intervals according to their left endpoints.
For $j\in \{1,\ldots, n\}$, let $v_j$ be the vertex represented by $I_j$.
It now suffices to show that for all $j\in \{2,\ldots, n-1\}$, the set $\{v_j\}$ is not a moplex in $H$.
Suppose towards a contradiction that $\{v_j\}$ is a moplex. 
Then, the graph $H-N[v_j]$ contains an $N(v_j)$-full component $C$.
Since $C$ contains no vertex from the closed neighbourhood of $v_j$, no interval representing a vertex in $C$ intersects $I_j$. 
Furthermore, since $C$ is a connected graph, we may assume that all its vertices are represented by intervals whose left endpoints are strictly larger than the right endpoint of $I_j$. 
The connectedness of $H$ and the ordering of the intervals imply that $I_{j-1}$ intersects $I_j$, that is, $v_{j-1}$ is adjacent to $v_j$.
However, since all intervals representing vertices in $C$ are disjoint from $I_j$ and lie entirely to the right of $I_j$, the fact that  $I_{j-1}$ ends before $I_j$ ends implies that no vertex in $C$ can be adjacent to $v_{j-1}\in N(v_j)$.
This contradicts the assumption that $C$ is an $N(v_j)$-full component of $H-N[v_j]$.
\end{proof}%
\renewcommand*{\proofname}{Proof.}

\noindent\textbf{Remark.} The fact that every connected proper interval graph has at most two moplexes can also be derived using known results from the literature:
\begin{enumerate}
	\item a result of Roberts~\cite{MR0252267} stating that 
every connected proper interval graph in which no two distinct vertices have the same closed neighbourhoods has at most two extreme vertices, where an extreme vertex is a simplicial vertex $s$ such that every pair of neighbours of $s$ have a common neighbour outside $N[s]$, and 
	\item the fact that for every minimal separator $S$ in a chordal graph $G$, every $S$-full component of $G-S$ has a vertex dominating $S$ (see, e.g.,~\cite{MR1663107}).
\end{enumerate}

\bigskip

At this point, let us recall the four hereditary graph classes which naturally correspond to any graph class, as mentioned in \cref{sec:intro}. Set $\moplexGraphs$ to be the class of $2$-moplex graphs and $\moplexGraphs_c$ to be the class of graphs all the connected components of which are $2$-moplex. Then the aforementioned four graph classes are:

\begin{description}
\item[class $\moplexGraphs^{-}$:]\label{item Gminus}
the largest hereditary graph class contained in $\moplexGraphs$, that is, the class of graphs all of whose induced subgraphs are $2$-moplex graphs,
\item[class $\moplexGraphs^{+}$:]\label{item Gplus} the smallest hereditary graph class containing $\moplexGraphs$, that is, the class of all induced subgraphs of $2$-moplex graphs,
\item[class $\moplexGraphs_c^{-}$:] \label{item Gcminus}
the largest hereditary graph class contained in $\moplexGraphs_c$, that is, the class of graphs all of whose induced subgraphs only have $2$-moplex graphs as connected components,
and
\item[class $\moplexGraphs_c^{+}$:]\label{item Gcplus}
the smallest hereditary graph class containing $\moplexGraphs_c$, that is, the class of all induced subgraphs of graphs in $\moplexGraphs_c$.
\end{description}

\medskip
The following inclusion relations follow directly from the definitions:
\begin{equation*}
\moplexGraphs^{-}\subseteq\: \moplexGraphs \subseteq\:\moplexGraphs^{+}\,, \quad\quad  
\moplexGraphs_c^{-}\subseteq\: \moplexGraphs_c \subseteq\:\moplexGraphs_c^{+},
\end{equation*}
\begin{equation*}
\moplexGraphs^{-}\subseteq\: \moplexGraphs_c^{-}, \quad
\moplexGraphs\subseteq\: \moplexGraphs_c, \quad
    \textrm{ and }
  \quad
\moplexGraphs^{+}\subseteq\: \moplexGraphs_c^{+}.
\end{equation*}

\medskip
\Cref{thm:hereditary-connected-M2}
implies the following.

\begin{corollary}
$\moplexGraphs_c^{-}$ is the class of proper interval graphs.
\end{corollary}

\begin{proof}
By \Cref{thm:hereditary-connected-M2}, it suffices to show that a graph $G$ belongs to $\moplexGraphs_c^{-}$ if and only if each connected induced subgraph of $G$ has at most two moplexes.

Assume first that $G$ is a graph in $\moplexGraphs_c^{-}$ and let $G'$ be a connected induced subgraph of $G$. Since $G\in \moplexGraphs_c^{-}$, each connected component of $G'$ is a $2$-moplex graph, which implies that $G'$, being connected, has at most two moplexes.

Assume now that each connected induced subgraph of $G$ has at most two moplexes.
To show that $G$ belongs to $\moplexGraphs_c^{-}$, consider an arbitrary induced subgraph $G'$ of $G$ and any component $C$ of $G'$.
Then $C$ is a connected induced subgraph of $G$ and thus a $2$-moplex graph by assumption.
Hence, $G$ belongs to $\moplexGraphs_c^{-}$.
\end{proof}

We now determine the remaining three hereditary classes, $\moplexGraphs^{-}$, $\moplexGraphs^{+}$, and $\moplexGraphs_c^{+}$, naturally associated with the class of $2$-moplex graphs.
\Cref{thm:cochain} shows that $\moplexGraphs^{-}$ is the class of cochain graphs, while the results in the following subsection show that $\moplexGraphs_c^{+}$ and $\moplexGraphs^{+}$ both coincide with the class of cocomparability graphs, yielding 

\begin{equation}\label{eq:subset-seq}
\underbrace{\moplexGraphs^{-}}_{\mathclap{\text{cochain}}} \subset\:  \overbrace{\moplexGraphs}^{\mathclap{\text{$2$-moplex}}} \subset\:   \underbrace{\moplexGraphs^{+}}_{\mathclap{\text{cocomparability}}}
  \quad\quad\quad
  \textrm{ and }
  \quad\quad\quad
\underbrace{\moplexGraphs_c^{-}}_{\mathclap{\text{proper interval}}} \subset\:  \overbrace{\moplexGraphs_c}^{\mathclap{\text{each component is $2$-moplex}}} \subset\:   \underbrace{\moplexGraphs_c^{+}}_{\mathclap{\text{cocomparability}}}\,.
\end{equation}

\Cref{fig:hasse} shows a Hasse diagram summarising both chains of inclusions, for the class of $2$-moplex graphs and in general.

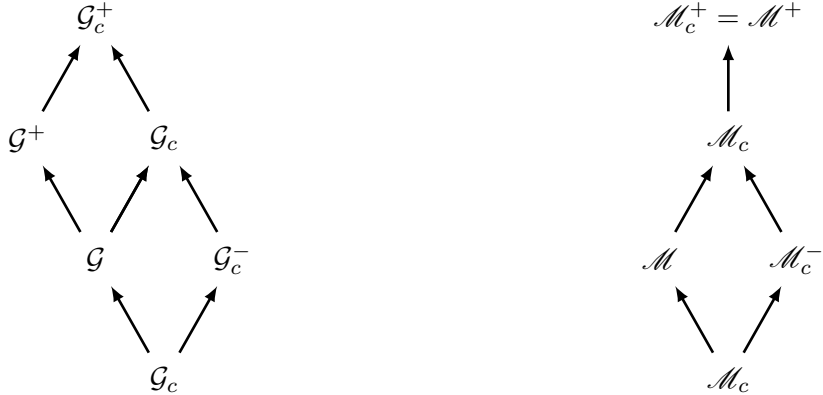
\begin{figure}[ht!]
    \begin{subfigure}[b]{0.495\textwidth}
        \centering
        \begin{tikzpicture}
        \tikzstyle{edge} = [draw=black,line width=1pt]
        \tikzstyle{directededge} = [edge,->,-latex]
        \def\hor{1.8}
        \def\vert{1.6}
        \def\height{20pt}
        \node (left) at (-4,0) {};
        \node (center) at (0,0) {};
        \node (right) at (4,0) {};
        \node[minimum height=\height] (M) at ($(center)+(-0.5*\hor,0)$) {$\mathcal{G}$};
        \node[minimum height=\height] (M-c-m) at ($(center)+(0.5*\hor,0)$) {$\mathcal{G}_c^{-}$};
        \node[minimum height=\height] (M-c) at ($(center)+(0,\vert)$) {$\mathcal{G}_c$};
        \node[minimum height=\height] (M-m) at ($(center)+(0,-\vert)$) {$\mathcal{G}_c$};
        \node[minimum height=\height] (M-c-p) at ($(M-c)+(-0.5*\hor,\vert)$) {$\mathcal{G}_c^{+}$};
        \node[minimum height=\height] (M-p) at ($(M-c)+(-\hor,0)$) {$\mathcal{G}^{+}$};
        
        \draw[directededge] (M-m) to (M);
        \draw[directededge] (M-m) to (M-c-m);
        \draw[directededge] (M-c-m) to (M-c);
        \draw[directededge] (M) to (M-c);
        \draw[directededge] (M) to (M-p);
        \draw[directededge] (M-c) to (M-c-p);
        \draw[directededge] (M) to (M-c);
        \draw[directededge] (M-p) to (M-c-p);
    \end{tikzpicture}
    \end{subfigure}%
    \hfill%
    \begin{subfigure}[b]{0.495\textwidth}
    \centering
        \begin{tikzpicture}
        \tikzstyle{edge} = [draw=black,line width=1pt]
        \tikzstyle{directededge} = [edge,->,-latex]
        \def\hor{1.8}
        \def\vert{1.6}
        \def\height{20pt}
        \node (left) at (-4,0) {};
        \node (center) at (0,0) {};
        \node (right) at (4,0) {};
        \node[minimum height=\height] (M) at ($(center)+(-0.5*\hor,0)$) {$\moplexGraphs$};
        \node[minimum height=\height] (M-c-m) at ($(center)+(0.5*\hor,0)$) {$\moplexGraphs_c^{-}$};
        \node[minimum height=\height] (M-c) at ($(center)+(0,\vert)$) {$\moplexGraphs_c$};
        \node[minimum height=\height] (M-m) at ($(center)+(0,-\vert)$) {$\moplexGraphs_c$};
        \node[minimum height=\height] (M-p) at ($(M-c)+(0,\vert)$) {$\moplexGraphs_c^{+}=\moplexGraphs^{+}$};
        
        \draw[directededge] (M-m) to (M);
        \draw[directededge] (M-m) to (M-c-m);
        \draw[directededge] (M-c-m) to (M-c);
        \draw[directededge] (M) to (M-c);
        \draw[directededge] (M-c) to (M-p);
    \end{tikzpicture}
    \end{subfigure}
    \caption{On the left, the Hasse diagram showing inclusion relations between several classes derived from an arbitrary graph class $\mathcal{G}$ (see \cpageref{item cochain} for the definitions). 
    When $\mathcal{G}$ is the class $\moplexGraphs$ of $2$-moplex graphs, the diagram simplifies to the diagram on the right.}
    \label{fig:hasse}
\end{figure}

Let us also comment on the remaining three inclusions 
$\moplexGraphs^{-}\subseteq\: \moplexGraphs_c^{-}$, 
$\moplexGraphs\subseteq\: \moplexGraphs_c$, and 
$\moplexGraphs^{+}\subseteq\: \moplexGraphs_c^{+}$.
Since $3K_1$ is a proper interval graph but not a $2$-moplex graph, the inclusions $\moplexGraphs^{-}\subseteq \moplexGraphs_c^{-}$ and $\moplexGraphs\subseteq \moplexGraphs_c$ are both proper.
On the other hand, the inclusion $\moplexGraphs^{+}\subseteq \moplexGraphs_c^{+}$ holds with equality.

\CochainsMaximalHereditarySubclass*

\begin{proof}
Suppose first that each induced subgraph of $G$ has at most two moplexes.
Since for every $k\ge 4$ the cycle $C_k$ is a connected graph in which every vertex forms a moplex, $G$ is a chordal graph.
Furthermore, since the graph $3K_1$ has three moplexes, $G$ has independence number at most two.
As $G$ is a chordal graph, it is also perfect, and hence the vertex set of $G$ can be covered with two disjoint cliques $X$ and $Y$ (see~\cite{MR302480}), that is, $G$ is cobipartite.
Furthermore, for any two vertices $x_1,x_2$ in $X$ we must have $N(x_1)\cap Y\subseteq N(x_2)\cap Y$ or $N(x_2)\cap Y\subseteq N(x_1)\cap Y$, since otherwise $G$ would contain an induced $4$-cycle. 
Using the fact that $X$ is a clique, we infer that $N[x_1]\subseteq N[x_2]$ or $N[x_2]\subseteq N[x_1]$; since this holds for any two vertices in $X$,
we conclude that $G$ is a cochain graph.

For the converse direction, let $G$ be a cochain graph.
We need to show that $G$ has at most two moplexes.
If $G$ is disconnected, then $G$ is isomorphic to the disjoint union of two complete graphs, and hence has moplex number two.
So we may assume that $G$ is connected.
It is not difficult to see that $G$ is a chordal graph.
Furthermore, since $G$ is a cobipartite graph, $G$ is also claw-free and AT-free.
Thus, $G$ is a proper interval graph by \cref{thm:PIGs}.
Finally, by \cref{thm:hereditary-connected-M2}, $G$ has at most two moplexes.
\end{proof}

\subsection{Connection with cocomparability graphs}
\label{sub:conwithcocomp}

In this subsection, we show that the smallest hereditary graph class containing the class of $2$-moplex graphs is the class of cocomparability graphs.
In order to prove that every $2$-moplex graph $G$ is a cocomparability graph, we identify a property common to all minimal $x{,}y$-separators, for any two non-adjacent vertices $x$ and $y$ of $G$.
We then exploit this property to orient the edges of the complement of $G$ in a transitive way.

Let $G$ be a non-complete $2$-moplex graph with the two moplexes $U,W\subseteq V(G)$.
For two non-adjacent vertices $x$ and $y$ we denote by $\mathcal{S}_G(x,y)$ (or simply $\mathcal{S}(x,y)$ if the graph is clear from the context) the set of all minimal $x{,}y$-separators in $G$. 
Given $M\in \{U,W\}$ and $S\in \mathcal{S}(x,y)$, we say that \emph{$M$ prefers $x$ to $y$ with respect to $S$} if $M$ and $x$ lie in the same connected component of $G-S$. 
By \cref{G-S two cc} of \cref{M2 moplexes properties}, either $M$ prefers $x$ to $y$ with respect to $S$ or $M$ prefers $y$ to $x$ with respect to $S$.
As the following key lemma shows, which of these two cases occurs is actually independent of the choice of~$S$.

\begin{lemma}
\label{lem:unique-separation}
Let $G$ be a non-complete $2$-moplex graph and denote by $U$ and $W$ its two moplexes.
Then, for each $M \in \Set{U,W}$ and for every two non-adjacent vertices $x$ and $y$ in $G$ exactly one of the following conditions holds:
\begin{itemize}
    \item $M$ prefers $x$ to $y$ with respect to all $S\in \mathcal{S}(x,y)$, or
    \item $M$ prefers $y$ to $x$ with respect to all $S\in \mathcal{S}(x,y)$.
\end{itemize}
\end{lemma}

\begin{proof}
Without loss of generality, assume that $M=U$.
Suppose towards a contradiction that there exist two minimal $x{,}y$-separators $S$ and $S'$
such that $U$ and $x$ lie in the same component of $G-S$, and $U$ and $y$ lie in the same component of $G-S'$.
Then, due to \cref{M2 moplexes properties}, $W$ and $y$ lie in the same component of $G-S$, and $W$ and $x$ lie in the same component of $G-S'$.
Clearly, $x\notin U$, as $x$ and $U$ lie in different components of $G-S'$.
In a similar way we conclude that neither $x$ nor $y$ can belong to $U\cup W$. 
Now fix any $u\in U$ and $w\in W$. It is easy to observe that $\{x,u,w\}$ is an independent set in $G$. 
We conclude the proof by deriving a contradiction with \cref{graphs in M2 are AT-free}. 
To this end, it is enough to show that $\{x,u,w\}$ is an asteroidal triple in $G$. 

The removal of $N[u]$ does not affect the component of $G-S'$ containing both $x$ and $w$. 
Similarly, the removal of $N[w]$ does not affect the component of $G-S$ containing both
$x$ and $u$. Finally, consider the graph $G-N[x]$ and first observe that it contains a $u{,}y$-path, 
as the removal of $N[x]$ does not affect the component of $G-S'$ 
containing $y$ and $u$. Similarly, $G-N[x]$ must also contain a $y{,}w$-path, 
as the removal of $N[x]$ does not affect the component of $G-S$ containing $y$ and $w$.
Thus, the vertices $u$ and $w$ are in the same component of $G-N[x]$ and $\{x,u,w\}$ is an asteroidal triple, as claimed.
\end{proof}

It follows from~\cref{lem:unique-separation} that if for \emph{some} minimal $x{,}y$-separator the vertices $x$ and $y$ belong to components containing the moplexes $U$ and $W$, respectively, then in fact for \emph{every} minimal $x{,}y$-separator $x$ and $y$ belong to components containing the moplexes $U$ and $W$, respectively.
Let $M\in \{U,W\}$.
If $M$ prefers $x$ to $y$ with respect to all $S\in \mathcal{S}(x,y)$, we say that \emph{$M$ prefers $x$ to~$y$}.
Using this we define a binary relation $R_M$ over $V(G)$ as follows:
\begin{align*}
    xR_{M}y & \iff M\text{ prefers }x\text{ to }y.
\end{align*}

By \cref{lem:unique-separation}, relations $R_U$ and $R_W$ are well-defined.
Furthermore, either $U$ prefers $x$ to $y$ or $U$ prefers $y$ to $x$ (in which case $W$ prefers $x$ to $y$).
We thus have $xR_Uy \iff yR_Wx$, that is, $R_{W}=R_{U}^{-1}$.
By definition, $xR_{U}y$ (or $yR_Wx$) implies that $x$ and $y$ are distinct and non-adjacent, and that $\mathcal S(x,y)\neq \emptyset$.  
This in turn implies that
\begin{align}
    U\cap N(y)=\emptyset=W\cap N(x),\label{eq:isolatedMoplex}
\end{align}
and that there are no edges between $U$ and $W$.

\begin{lemma}
\label{lem:transitive}
Let $G$ be a non-complete $2$-moplex graph and denote by $U$ and $W$ its two moplexes.
Then, relations $R_{U}$ and $R_{W}$ are transitive.
\end{lemma}

\begin{proof}
Let $x,y,z$ be vertices such that $U$ prefers $x$ to $y$ as well as $y$ to $z$, which can be written as $xR_{U}y$ and $zR_{W}y$.
Together with \cref{eq:isolatedMoplex} and the fact that there are no edges between $U$ and $W$, we infer that $\{u,w,y\}$ is an independent set for all $u\in U$ and $w\in W$.
We first show that if $xz\in E(G)$, then $\{u,w,y\}$ must be an asteroidal triple.

Notice that $yR_{U}z$ implies that the removal of $N[w]$ preserves a $u,y$-path. 
Similarly, $xR_{U}y$ implies that $yR_{W}x$, and thus the removal of $N[u]$ preserves a $y,w$-path.
The graph $G-N[y]$ contains a $u,x$-path, since $xR_{U}y$, and, similarly, $G-N[y]$ contains a $z,w$-path, since $zR_{W}y$.
Suppose $xz\in E(G)$, then it connects both paths to a $u{,}w$-path in $G-N[y]$, yielding that  $\{u,w,y\}$ is an asteroidal triple.
However, according to \cref{graphs in M2 are AT-free}, $2$-moplex graphs are AT-free, a contradiction.
Thus, $\{x,y,z\}$ is an independent set.

To prove the transitivity of $R_U$, it suffices to show that $xR_Uz$. 
Note first that vertices $x$ and $z$ are distinct, since otherwise 
conditions $xR_Uy$ and $yR_Ux$ would hold simultaneously.
Suppose for a contradiction that $U$ does not prefer $x$ to $z$.
Then, since $x$ and $z$ are non-adjacent, $U$ prefers $z$ to $x$. 
Furthermore, since $\{x,y,z\}$ is an independent set, there exists a minimal $x{,}z$-separator $S$ such that $y\not\in S$.
Recall that $U$ and $W$ lie in the same component of $G-S$ as $z$ and $x$, respectively. 
Now notice that since $S$ separates $y$ from either $x$ or $z$,
it follows that some $S'\subseteq S$ is either a minimal $x{,}y$-separator, or a minimal $y{,}z$-separator.

Suppose first that $S'$ is a minimal $x{,}y$-separator.
As $zR_{U}x$ and $S'\subseteq S$, the vertex $x$ remains in the same component together with $W$ in $G-S'$, and thus $xR_{W}y$. 
But then $xR_{U}y$ violates the fact that $R_{W}=R_{U}^{-1}$, a contradiction. 
The case when $S'$ is a minimal $y{,}z$-separator is similar.
This concludes the proof that $xR_{U}y$ and $yR_{U}z$ indeed imply $xR_{U}z$. 
Therefore, the relation $R_U$ is transitive.
By symmetry, so is $R_W$.
\end{proof}

We remark that $R_U$ is a strict partial order on the vertices of $G$. Furthermore, since $R_U$ is an orientation of the edges of the complement of $G$, \cref{lem:transitive} implies the main result of this section.
 
\begin{proposition}\label{M2 is cocomp}
Every $2$-moplex graph is a cocomparability graph.
\end{proposition}
\begin{proof}
Consider a $2$-moplex graph $G$. If $G$ is a complete graph, then $G$ is cocomparability.
Otherwise, $G$ has exactly two moplexes $U$ and $W$.
Consider the orientation of the edges of the complement of $G$ obtained by orienting each edge $\Set{x,y}\in E(\overline{G})$ from $x$ to $y$ if and only if $xR_Uy$. 
By \cref{lem:unique-separation}, this orientation is well-defined.
Furthermore, by \cref{lem:transitive}, it is a transitive orientation.
Thus, $G$ is a cocomparability graph.
\end{proof}

\Cref{M2 is cocomp} is a strengthening of \cref{graphs in M2 are AT-free}, with algorithmic consequences for the class of \hbox{$2$-moplex} graphs.
First, \textsc{Weighted Independent Set} is solvable in linear time in the class of cocomparability graphs~\cite{MR3466620}, and thus also in the class of $2$-moplex graphs. In the more general class of AT-free graphs, this problem is only known to be solvable in time $\mathcal{O}(|V(G)|^3)$~\cite{EkkiThesis}.
Furthermore, \Cref{M2 is cocomp} implies that the class of $2$-moplex graphs is a subclass of the class of perfect graphs, and thus \textsc{Clique} (and its weighted generalisation),  \textsc{Clique Cover}, and \textsc{Colouring} are all solvable in polynomial time in the class of $2$-moplex graphs~\cite{MR936633}.
Note that this conclusion cannot be derived from \cref{graphs in M2 are AT-free}: since \textsc{Independent Set} and \textsc{Colouring} are \NP-hard in the class of $C_3$-free graphs (see~\cite{MR2024261,MR1905637}), \textsc{Clique} and \textsc{Clique Cover} are \NP-hard in the class of $3K_1$-free graphs (and thus in the more general class of AT-free graphs), while the complexity of \textsc{Colouring} is still open in the class of AT-free graphs (see~\cite{MR1686810,kratsch2012colouring}).

As we show next, the result of \cref{M2 is cocomp} is best possible.

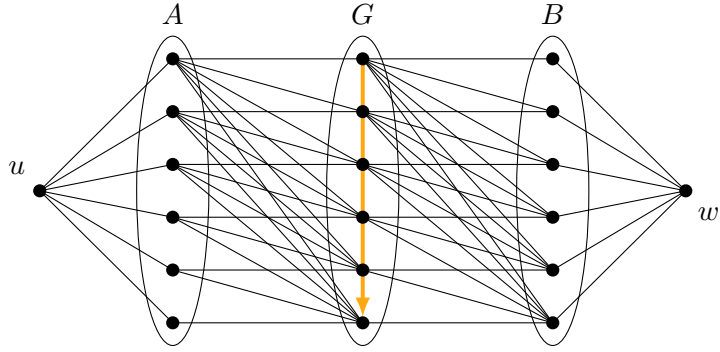
\begin{figure}[ht]
	\centering
	\begin{tikzpicture}
		\pgfdeclarelayer{background}
		\pgfdeclarelayer{foreground}
		
		\pgfsetlayers{background,main,foreground}
		
		\tikzstyle{vertex} = [draw, circle, scale=0.4, thick,fill]
		\tikzstyle{Nonvertex}=[circle,
		inner sep=0pt,
		text width=2.5mm,
		align=center,
		draw=black,
		transform shape,
		fill=white]
		\def\vert{0.7}
		\def\hor{2.5}
		\foreach \i in {1,...,6}
		{
			\node[vertex] (v-\i) at ($(0,-\i*\vert)$) {};
			\node[vertex] (a-\i) at ($(-\hor,-\i*\vert)$) {};
			\node[vertex] (b-\i) at ($(\hor,-\i*\vert)$) {};
		}
		\node[vertex] (u) at ($(-1.7*\hor,-3.5*\vert)$) {};
		\node[vertex] (w) at ($(1.7*\hor,-3.5*\vert)$) {};
		
		\node (u-label) at ($(u)+(-0.3,0.3)$) {$u$};
		\node (w-label) at ($(w)+(0.3,-0.3)$) {$w$};
		
		\begin{pgfonlayer}{background}
			\node[Nonvertex,ellipse,minimum width=0.95cm,text height=2.9cm] at ($(-\hor,-3.5*\vert)$) {};
			\node[Nonvertex,ellipse,minimum width=0.95cm,text height=2.9cm] at ($(0,-3.5*\vert)$) {};
			\node[Nonvertex,ellipse,minimum width=0.95cm,text height=2.9cm] at ($(\hor,-3.5*\vert)$) {};
		\end{pgfonlayer}
		
		\foreach \i in {1,...,6}
		{
			\foreach \j in {1,...,6}
			{
				\ifthenelse{\i < \j \OR \equal{\i}{\j}}{
					\draw (a-\i) -- (v-\j);
					\draw (v-\i) -- (b-\j);
				}{}
			}
			\draw (u) -- (a-\i);
			\draw (w) -- (b-\i);
		}
	
		\begin{pgfonlayer}{background}
			\draw[myOrange,->,-latex,line width=1.5pt] (v-1) -- (v-6);
		\end{pgfonlayer}
	
		\node (A) at ($(a-1)-(0,-0.6)$) {$A$};
		\node (B) at ($(b-1)-(0,-0.6)$) {$B$};
		\node (G) at ($(v-1)-(0,-0.6)$) {$G$};
	\end{tikzpicture}
	\caption{The graph $G'$ constructed from the graph $G$ whose vertices are arranged according to a cocomparability ordering.}
	\label{fig:cocomp_subgraph}
\end{figure}

\begin{proposition}
\label{thm:cocomp-are-induced-in-M2}
Every cocomparability graph is an induced subgraph of some connected $2$-moplex graph.
\end{proposition}
\begin{proof}
Let $G$ be a cocomparability graph. 
By \cref{thm:umbrella}, $G$ has an umbrella-free ordering $\sigma$.
Let us write $\sigma = \Brace{v_1,\ldots, v_n}$ where $i<j$ if and only if 
$v_i <_{\sigma} v_j$.
Consider the graph $G'$ obtained from $G$ as follows (see \cref{fig:cocomp_subgraph} for an illustration of the construction):
\begin{itemize}
\item add a set $A = \{a_1,\ldots, a_n\}$ of vertices and a vertex $u$ such that 
$A\cup \{u\}$ is a clique;
\item add a set $B = \{b_1,\ldots, b_n\}$ of vertices and a vertex $w$ such that $B\cup \{w\}$ is a clique;
\item for all $i,j\in \{1,\ldots, n\}$ such that $i\le j$, 
add an edge from $a_i$ to $v_j$ 
and an edge from $v_i$ to $b_j$.
\end{itemize}
Clearly, $G$ is an induced subgraph of $G'$ and $G'$ is connected. 
It is not difficult to see that $U = \{u\}$ and $W = \{w\}$ are moplexes in $G'$.
To complete the proof, it suffices to show that $G'$ has no other moplexes.
Since no two distinct vertices of $G'$ have the same closed neighbourhoods, every moplex in $G'$ consists of a single vertex. 

We show that for all $v\in V(G')\setminus\{u,w\}$, the set $\{v\}$ is not a moplex in $G'$.
By symmetry, we may assume that $v\in \{a_i,v_i\}$ for some $i\in \{1,\ldots, n\}$.
Suppose first that $v = a_i$. 
Since $u$ is a neighbour of $a_i$ not adjacent to any vertex in $G'-N[a_i]$, no component of $G'-N[a_i]$ can be $N(a_i)$-full.
Thus, $\{a_i\}$ is not a moplex.
Suppose now that $v = v_i$.
To show that $\{v_i\}$ is not a moplex, we need to verify that no component of $G'-N[v_i]$ is $N(v_i)$-full.
Consider the vertex
sets $$X = \{u\}\cup \{a_j\mid i< j\le n\} \cup (\{v_j\mid i< j\le n\}\setminus N(v_i))$$ and
$$Y = (\{v_j\mid 1\le j<i\}\setminus N(v_i))\cup \{b_j\mid 1\le j<i\}\cup\{w\}\,,$$ and let $C$ and $D$ denote the subgraphs of $G'-N[v_i]$ induced by $X$ and $Y$, respectively. Then $C$ and $D$ are connected. Note that $G'$ contains no edges from a vertex in $A\cup \{u\}$ to a vertex in 
$B\cup \{w\}$ and also no edge from a vertex in 
$\{v_j\mid 1\le j<i\}\setminus N(v_i)$
to a vertex in 
$\{v_j\mid i< j\le n\}\setminus N(v_i)$, since $(v_1,\ldots, v_n)$ is an umbrella-free ordering of $G$.
It follows that the graph $G'-N[v_i]$ has exactly two components, namely $C$ and $D$. By construction, $C$ contains no vertex adjacent to $b_i$ and 
$D$ contains no vertex adjacent to $a_i$. Since $a_i$ and $b_i$ are adjacent to $v_i$, we infer that $\{v_i\}$ is not a moplex.
Therefore, $G'$ is connected $2$-moplex, as claimed.
\end{proof}

Let us now explain how the above results establish that both $\moplexGraphs^{+}$ and $\moplexGraphs_c^{+}$ coincide with the class of cocomparability graphs, as stated in \cref{eq:subset-seq}.

\begin{corollary}
$\moplexGraphs^{+}=\moplexGraphs_c^{+}=\mathcal{C}$, where $\mathcal{C}$ is the class of cocomparability graphs.
\end{corollary}

\begin{proof}
Since $\moplexGraphs^{+}\subseteq\moplexGraphs_c^{+}$, it suffices to show that $\mathcal{C}\subseteq \moplexGraphs^{+}$ and $\moplexGraphs_c^{+}\subseteq \mathcal{C}$.

The inclusion $\mathcal{C}\subseteq \moplexGraphs^{+}$ is an immediate consequence of \Cref{thm:cocomp-are-induced-in-M2}.
For the inclusion $\moplexGraphs_c^{+}\subseteq \mathcal{C}$, consider an arbitrary graph $G$ in $\moplexGraphs_c^{+}$.
Then there exists a graph $G'$ such that every connected component of $G'$ is a $2$-moplex graph and $G$ is an induced subgraph of $G'$.
By \Cref{M2 is cocomp}, every connected component of $G'$ is a cocomparability graph.
\Cref{thm:umbrella} implies that the class of cocomparability graphs is closed under disjoint union.
It follows that $G'$ is a cocomparability graph, and, since the class of  cocomparability graphs is hereditary, so is $G$.
We conclude that $\moplexGraphs_c^{+}\subseteq \mathcal{C}$, as desired.
\end{proof}

In particular, we have the following result announced in the introduction.

\CocompIsMinimalHereditarySuperclass*

\section{Hardness results}
\label{sec:compl}

As seen in the last section, the class of $2$-moplex graphs is a proper subclass of the class of cocomparability graphs.
In this section we show that two classical problems, namely \textsc{Max-Cut} and \textsc{Graph Isomorphism}, remain as hard on the class of $2$-moplex graphs as they are on cocomparability graphs.

\subsection{Hardness of Max-Cut on 2-moplex graphs}

A \emph{cut} $\Brace{Z_1,Z_2}$ of a graph $G$ is an ordered partition of $V(G)$ into two parts, $Z_1$ and $Z_2$. 
An edge of $G$ is a \emph{cut edge} if its endpoints are in different parts of the cut, that is, if one endpoint is in $Z_1$ and the other in $Z_2$.
The \emph{size} of a cut corresponds to the number of cut edges.

The \textsc{Max-Cut} problem is defined as follows.

\problemdef[max-cut-def]{Max-Cut}{An undirected graph $G$ and $k \in \N$.}{Does $G$ contain a cut of size at least $k$?}

Recall that we have established that the class of connected $2$-moplex graphs is sandwiched between the classes of connected proper interval and cocomparability graphs.

It is known that \textsc{Max-Cut} is \NP-complete on cobipartite graphs \cite{MR1769833}, and thus on cocomparability graphs. 
Interestingly, the complexity of \textsc{Max-Cut} is still open on proper interval graphs~\cite{bodlaender2004simple,MR4340841,MR4287014,DBLP:conf/mfcs/FigueiredoMO021}.

We show that the problem remains \NP-complete on cobipartite graphs with only two moplexes. 
The hardness reduction is based on the following construction  (see also \cref{fig:maxcut}).

\begin{construction}\label{cobipartite to A_2}
    Let $G = (A \cup B, E)$ be a cobipartite graph such that $A$ and $B$ are disjoint cliques.
    We define the graph $G'$ obtained from $G$ as follows:
    \begin{itemize}
        \item add a set $A'$ containing $|A|$ vertices and a vertex $u$ such that $\{u\} \cup A \cup A'$ is a clique;
        \item add a set $B'$ containing $|B|$ vertices and a vertex $w$ such that $\{w\} \cup B \cup B'$ is a clique;
        \item fix a vertex $a^* \in A'$ and connect it to every vertex in $B \cup B'$;
        \item fix a vertex $b^* \in B'$ and connect it to every vertex in $A \cup A'$.
    \end{itemize}
\end{construction}

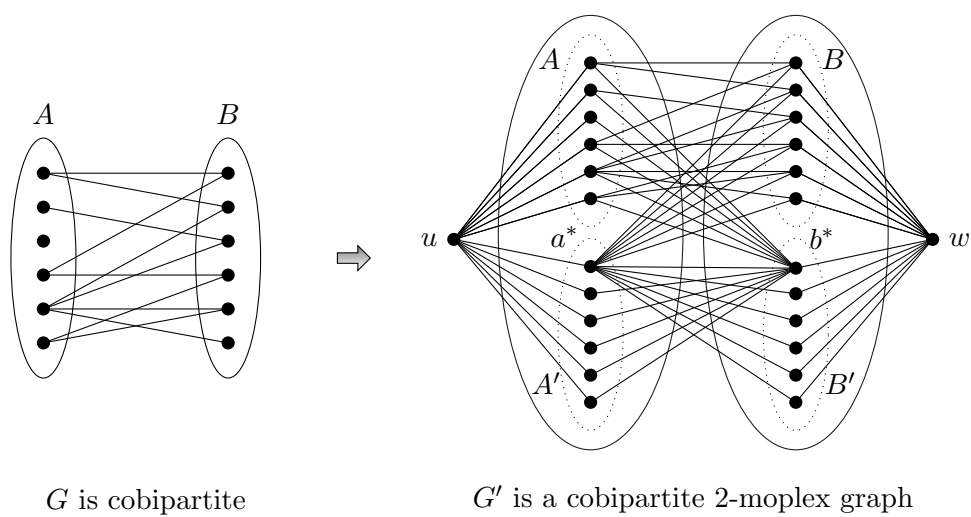
\begin{figure}[thb]
	\centering
	\usetikzlibrary{shapes,backgrounds}
\begin{tikzpicture}[scale=0.9]
\tikzstyle{Nonvertex}=[circle,
				inner sep=0pt,
				text width=2.5mm,
				align=center,
				draw=black,
				transform shape,
				fill=white]
\tikzstyle{vertex} = [draw, circle, scale=0.4, thick,fill]
\useasboundingbox (-1,-4) rectangle (14.5,3.7);
\begin{scope}
    \node[Nonvertex,rectangle,draw=none,fill=none,text width=3cm,label={$G$ is cobipartite}] at (1.5,-3.95) {};
    \begin{scope}
    	\node[Nonvertex,ellipse,minimum width=0.95cm,text height=2.5cm] at (0,0) {};
    	\node[Nonvertex,draw=none,fill=none,rectangle,text width=2cm,label={\mathstrut{$A$}}] at (0,1.7) {};
    	\node[vertex] (a2) at (0,1.25) {};
    	\node[vertex] (a3) at (0,.75) {};
    	\node[vertex] (a4) at (0,.25) {};
    	\node[vertex] (a5) at (0,-.25) {};
    	\node[vertex] (a6) at (0,-.75) {};
    	\node[vertex] (a7) at (0,-1.25) {};
    \end{scope}
    \begin{scope}[xshift=2.7cm]
    	\node[Nonvertex,ellipse,minimum width=0.95cm,text height=2.5cm] at (0,0) {};
    	\node[Nonvertex,draw=none,fill=none,rectangle,text width=2cm,label={\mathstrut{$B$}}] at (0,1.7) {};
    	\node[vertex] (b2) at (0,1.25) {};
    	\node[vertex] (b3) at (0,.75) {};
    	\node[vertex] (b4) at (0,.25) {};
    	\node[vertex] (b5) at (0,-.25) {};
    	\node[vertex] (b6) at (0,-.75) {};
    	\node[vertex] (b7) at (0,-1.25) {};
    \end{scope}
    \draw (a2) -- (b2);
    \draw (a2) -- (b3);
    \draw (a6) -- (b3);
    \draw (a3) -- (b4);
    \draw (a5) -- (b5);
    \draw (a5) -- (b2);
    \draw (a6) -- (b4);
    \draw (a6) -- (b6);
    \draw (a6) -- (b7);
    \draw (a7) -- (b5);
    \draw (a7) -- (b6);
\end{scope}
\node[Nonvertex,
    single arrow,                  %
    single arrow head extend=3pt,  %
    draw,                          %
    inner sep=2pt,                 %
    top color=gray,               %
    bottom color=white,               %
  ] at (4.5,0) {};%
\begin{scope}[xshift=8cm,yshift=1.875cm,yscale=0.8]
    \node[Nonvertex,rectangle,draw=none,fill=none,text width=3cm,label={$G'$ is a cobipartite \mbox{$2$-moplex} graph}] at (1.5,-7.25) {};
    \draw[color=gray!20!black] (0,-1.875) ellipse (1.35 and 4);
    \draw[color=gray!20!black] (3,-1.875) ellipse (1.35 and 4);
    \node[vertex,label=left:{$u$}] (u) at (-2,-2) {};
    \begin{scope}
    	\node[Nonvertex,ellipse,dotted,minimum width=0.95cm,text height=2.5cm] at (0,0) {};
    	\node[Nonvertex,draw=none,fill=none,rectangle,text width=2cm,label={\mathstrut{$A$}}] at (-0.6,0.8) {};
    	\node[vertex] (a2) at (0,1.25) {};
    	\node[vertex] (a3) at (0,.75) {};
    	\node[vertex] (a4) at (0,.25) {};
    	\node[vertex] (a5) at (0,-.25) {};
    	\node[vertex] (a6) at (0,-.75) {};
    	\node[vertex] (a7) at (0,-1.25) {};
    \end{scope}
    \begin{scope}[xshift=3cm]
        \node[vertex,label=right:{$w$}] (v) at (2,-2) {};
    	\node[Nonvertex,ellipse,dotted,minimum width=0.95cm,text height=2.5cm] at (0,0) {};
    	\node[Nonvertex,draw=none,fill=none,rectangle,text width=2cm,label={\mathstrut{$B$}}] at (0.55,0.8) {};
    	\node[vertex] (b2) at (0,1.25) {};
    	\node[vertex] (b3) at (0,.75) {};
    	\node[vertex] (b4) at (0,.25) {};
    	\node[vertex] (b5) at (0,-.25) {};
    	\node[vertex] (b6) at (0,-.75) {};
    	\node[vertex] (b7) at (0,-1.25) {};
    \end{scope}
    \draw (a2) -- (b2);
    \draw (a2) -- (b3);
    \draw (a6) -- (b3);
    \draw (a3) -- (b4);
    \draw (a5) -- (b5);
    \draw (a5) -- (b2);
    \draw (a6) -- (b4);
    \draw (a6) -- (b6);
    \draw (a6) -- (b7);
    \draw (a7) -- (b5);
    \draw (a7) -- (b6);
    \foreach \z in {2,3,4,5,6,7}{
        \draw (u) -- (a\z);
        \draw (v) -- (b\z);
    }
    \begin{scope}[yshift=-3.75cm]
        \begin{scope}
    	\node[Nonvertex,ellipse,dotted,minimum width=0.95cm,text height=2.5cm] at (0,0) {};
    	\node[Nonvertex,draw=none,fill=none,rectangle,text width=2cm,label={\mathstrut{$A'$}}] at (-0.65,-1.45) {};
        \node[vertex,label={[label distance=1.5pt]100:{$a^*$}}] (a'2) at (0,1.25) {};
    	\node[vertex] (a'3) at (0,.75) {};
    	\node[vertex] (a'4) at (0,.25) {};
    	\node[vertex] (a'5) at (0,-.25) {};
    	\node[vertex] (a'6) at (0,-.75) {};
    	\node[vertex] (a'7) at (0,-1.25) {};
    \end{scope}
    \begin{scope}[xshift=3cm]
        \node[Nonvertex,ellipse,dotted,minimum width=0.95cm,text height=2.5cm] at (0,0) {};
        \node[Nonvertex,draw=none,fill=none,rectangle,text width=2cm,label={\mathstrut{$B'$}}] at (0.65,-1.45) {};
    	\node[vertex,label={[label distance=1.4pt]70:{$b^*$}}] (b'2) at (0,1.22) {};
    	\node[vertex] (b'3) at (0,.75) {};
    	\node[vertex] (b'4) at (0,.25) {};
    	\node[vertex] (b'5) at (0,-.25) {};
    	\node[vertex] (b'6) at (0,-.75) {};
    	\node[vertex] (b'7) at (0,-1.25) {};
    \end{scope}
    \foreach \z in {2,3,4,5,6,7}{
        \draw (a'2) -- (b\z);
        \draw (b'2) -- (a\z);
        \draw (a'2) -- (b'\z);
        \ifnum \z = 2 %
        \else
        \draw (b'2) -- (a'\z);
        \fi
        \draw (u) -- (a\z);
        \draw (v) -- (b\z);
        \draw (u) -- (a'\z);
        \draw (v) -- (b'\z);
        }
    \end{scope}
\end{scope}
\end{tikzpicture}
	\caption{An example of \cref{cobipartite to A_2} used to prove \cref{thmMaxCutNPComplete}. The ellipses represent cliques.}\label{fig:maxcut}
\end{figure}

\begin{lemma}\label{G' is in A2}
    Let $G = (A \cup B, E)$ be a cobipartite graph such that $A$ and $B$ are disjoint cliques and $G'$ be the graph obtained from $G$ by \cref{cobipartite to A_2}.
    Then $G'$ is a cobipartite graph with exactly two moplexes.
\end{lemma}

\begin{proof}
    First, notice that the sets $A \cup A' \cup \{u\}$ and $B \cup B' \cup \{w\}$ are cliques and form a partition of the vertex set of $G'$.
    Thus, $G'$ is cobipartite.
    Furthermore, it is easily observed that $\Set{u}$ and $\Set{w}$ are simplicial moplexes.
    Now, fix a vertex $a \in A \cup A'$.
    First, observe that $N[u] \neq N[a]$, since $a$ is adjacent to $b^*$ but $u$ is not.
    Furthermore, since $u$ is a neighbour of $a$ and $N[u] \subseteq N[a]$, no component of $G' - N[a]$ can be $N(a)$-full.
    Hence, $a$ does not belong to any moplex.
    A similar argument can be used to show that every vertex $b \in B \cup B'$ does not belong to any moplex.
    Thus, $G$ contains exactly two moplexes.
\end{proof}

\MaxCutNPComplete*
\begin{proof}
    Clearly, the problem is in \NP.
    To prove \NP-hardness, we make a reduction from \textsc{Max-Cut} on cobipartite graphs, which is \NP-hard~\cite{MR1769833}. Let $G = (A \cup B, E)$ be a cobipartite graph such that $A$ and $B$ are disjoint cliques.
    Let $G'$ be the graph obtained from $G$ through \cref{cobipartite to A_2}. Note that $G'$ can be obtained in polynomial time.
    From \cref{G' is in A2}, we have that $G'$ is a cobipartite graph with exactly two moplexes.
    We complete the proof by showing that there exists a cut of size at least $k$ in $G$ if and only if there exists a cut of size at least $(|A|+1)^2+(|B|+1)^2 + k$ in $G'$.
    
    First, assume we have a cut $\Brace{Z_1,Z_2}$ of $G$ of size at least $k \geq 1$.
    Without loss of generality, we assume that $|A \cap Z_1| \geq 1$ and $|B \cap Z_2| \geq 1$.
    Indeed, if, say, $A \cap Z_1 = \emptyset$, then $Z_1\subseteq B$ and $A\subseteq Z_2$, in which case we can achieve the desired inequalities by swapping the roles of $Z_1$ and $Z_2$.
    We define a cut $\Brace{Z'_1,Z'_2}$ of $G'$ as follows.
    The first step is to choose a subset $A'_1 \subseteq A'\setminus\Set{a^*}$ such that $\Abs{\Brace{Z_1\cap A}\cup\Set{a^*}\cup A'_1} = \Abs{A}+1$ and a subset $B'_2 \subseteq B'\setminus\Set{b^*}$ such that $\Abs{\Brace{Z_2\cap B}\cup\Set{b^*}\cup B'_2} = \Abs{B}+1$.
    Then, we define $Z'_1$ and $Z_2'$ as follows:
    \begin{align*}
    Z_1' &= Z_1 \cup \Set{a^*,w} \cup A'_1 \cup \left(B'\setminus (B'_2 \cup \{b^*\})\right)\, ,\\
    Z'_2 &= Z_2 \cup \Set{b^*,u} \cup B'_2 \cup \left(A'\setminus (A'_1 \cup \{a^*\})\right)\, .
    \end{align*}
    This gives us that $|(A \cup A') \cap Z_1'| = |A| + 1$. 
    Using also the facts that $|A| = |A'|$, and that $Z_1$ and $Z_2$ form a partition of $V(G)$, we infer that $|(A \cup A') \cap Z_2'| = |A| - 1$.
    Similarly, we obtain that $|(B \cup B') \cap Z_1'| = |B| - 1$ and consequently $|(B \cup B') \cap Z_2'| = |B| + 1$.
    Now, we reason about the size of the cut $\Brace{Z'_1,Z'_2}$.
    Consider the set of cut edges with both endpoints in $A \cup A' \cup \{u, b^*\}$.
    It contains $2\Brace{\Abs{A}+1}$ edges between $A \cup A'$ and $\Set{u,b^*}$ and, since $A \cup A'$ is a clique, $\Brace{\Abs{A}+1} \Brace{\Abs{A}-1}$ edges within $A \cup A'$.
    Thus, there are exactly $2\Brace{\Abs{A}+1}+\Brace{\Abs{A}+1}\Brace{\Abs{A}-1}= (|A|+1)^2$ cut edges whose both endpoints lie in $A \cup A' \cup \{u, b^*\}$.
    Similarly, there are $(|B|+1)^2$ cut edges whose both endpoints lie in $B \cup B' \cup \{w, a^*\}$.
    Hence, adding the (at least $k$) cut edges between $A$ and $B$, we obtain that $\Brace{Z'_1,Z'_2}$ is a cut of size at least $(|A|+1)^2+(|B|+1)^2 + k$ in $G'$.
    
    Second, assume that we have a cut $\Brace{Z'_1,Z'_2}$ of $G'$ of size at least $(|A|+1)^2+(|B|+1)^2 + k$.
    Consider the cut $\Brace{Z_1,Z_2}$ of $G$ defined by $Z_1 = Z_1' \cap (A \cup B)$ and $Z_2 = Z_2' \cap (A \cup B)$. There are three possible cases:
    \begin{enumerate}
        \item\label[case]{u and b* in W1} $u, b^* \in Z_1'$;
        \item\label[case]{u in W1 and b* in W2} $u \in Z_1'$ and $b^* \in Z_2'$, or $u \in Z_2'$ and $b^* \in Z_1'$;
        \item\label[case]{u and b* in W2} $u, b^* \in Z_2'$.
    \end{enumerate}
    Consider the subgraph $H_A'$ of $G'$ induced by $ A \cup A' \cup \{u, b^*\}$ and let $p = |(A \cup A') \cap Z_1'|-|A|$;
    intuitively, $p$ captures the difference between how much $Z_1'$ intersects $A\cup A'$, and the size of the original set $A$.    
    Note that $|(A \cup A') \cap Z_1'| = |A|+p$ and $|(A \cup A') \cap Z_2'| = |A|-p$.
    In \cref{u and b* in W1}, the number of cut edges in $H_A'$ is exactly $(|A|+p+2)(|A|-p) = (|A|+1)^2-(p+1)^2$.
    In \cref{u in W1 and b* in W2}, the number of cut edges in $H_A'$ is exactly  $(|A|+p)(|A|-p) + (|A|+p) + (|A| - p) = (|A|+1)^2-(p^2+1)$.
    And finally, in \cref{u and b* in W2}, the number of cut edges in $H_A'$ is exactly $(|A|+p)(|A|-p+2) = (|A|+1)^2-(p-1)^2$.
    Thus, at most $(|A|+1)^2$ cut edges can exist in $H_A'$.
    A similar approach shows that the number of cut edges in the subgraph induced by the vertices in $B \cup B' \cup \{a^*,w\}$ is bounded by $(|B|+1)^2$.
    By assumption, the size of $\Brace{Z'_1,Z'_2}$ is at least $(|A|+1)^2+(|B|+1)^2 + k$, which implies that the other $k$ cut edges must be between the sets $A$ and $B$, and thus $\Brace{Z_1,Z_2}$ is a cut of size at least $k$ in $G$.
\end{proof}

\subsection{Hardness of Graph Isomorphism on 2-moplex graphs}

The \textsc{Graph Isomorphism} problem is defined as follows.

\problemdef[GI-def]{Graph Isomorphism}{Two graphs $G_1$ and $G_2$.}{Are $G_1$ and $G_2$ isomorphic to each other?}

The \textsc{Graph Isomorphism} problem is solvable in linear time in the class of interval graphs~\cite{MR528025}. Since the problem is \GI-complete on bipartite graphs \cite{uehara2005graph}, it is also \GI-complete on cobipartite graphs, and thus on cocomparability graphs. 
Using a reduction from the \textsc{Graph Isomorphism} problem in the class of bipartite graphs, we show that the problem remains hard on cobipartite graphs with at most two moplexes.

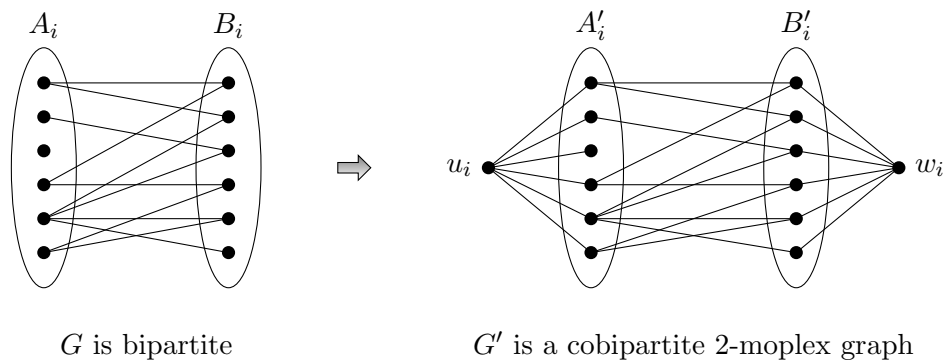
\begin{figure}[thb]
	\centering
	\usetikzlibrary{shapes,backgrounds}
\begin{tikzpicture}[scale=0.9]
\tikzstyle{Nonvertex}=[circle,
				inner sep=0pt,
				text width=2.5mm,
				align=center,
				draw=black,
				transform shape,
				fill=white]
\tikzstyle{vertex} = [draw, circle, scale=0.4, thick,fill]
\begin{scope}
    \node[Nonvertex,rectangle,draw=none,fill=none,text width=3cm,label={$G$ is bipartite}] at (1.5,-3) {};
    \begin{scope}
    	\node[Nonvertex,ellipse,minimum width=0.95cm,text height=2.5cm] at (0,0) {};
    	\node[Nonvertex,draw=none,fill=none,rectangle,text width=2cm,label={\mathstrut{$A_i$}}] at (0,1.7) {};
    	\node[vertex] (a2) at (0,1.25) {};
    	\node[vertex] (a3) at (0,.75) {};
    	\node[vertex] (a4) at (0,.25) {};
    	\node[vertex] (a5) at (0,-.25) {};
    	\node[vertex] (a6) at (0,-.75) {};
    	\node[vertex] (a7) at (0,-1.25) {};
    \end{scope}
    \begin{scope}[xshift=2.7cm]
    	\node[Nonvertex,ellipse,minimum width=0.95cm,text height=2.5cm] at (0,0) {};
    	\node[Nonvertex,draw=none,fill=none,rectangle,text width=2cm,label={\mathstrut{$B_i$}}] at (0,1.7) {};
    	\node[vertex] (b2) at (0,1.25) {};
    	\node[vertex] (b3) at (0,.75) {};
    	\node[vertex] (b4) at (0,.25) {};
    	\node[vertex] (b5) at (0,-.25) {};
    	\node[vertex] (b6) at (0,-.75) {};
    	\node[vertex] (b7) at (0,-1.25) {};
    \end{scope}
    \draw (a2) -- (b2);
    \draw (a2) -- (b3);
    \draw (a6) -- (b3);
    \draw (a3) -- (b4);
    \draw (a5) -- (b5);
    \draw (a5) -- (b2);
    \draw (a6) -- (b4);
    \draw (a6) -- (b6);
    \draw (a6) -- (b7);
    \draw (a7) -- (b5);
    \draw (a7) -- (b6);
\end{scope}
 \node[Nonvertex,
    single arrow,                  
    single arrow head extend=3pt,  
    draw,                          
    inner sep=2pt,                 
    top color=gray,               
    bottom color=white,               
  ] at (4.5,0) {};%
\begin{scope}[xshift=8cm]
    \node[Nonvertex,rectangle,draw=none,fill=none,text width=3cm,label={$G'$ is a cobipartite \mbox{$2$-moplex} graph}] at (1.5,-3) {};
    \node[vertex,label=left:{$u_i$}] (u) at (-1.5,0) {};
    \begin{scope}
    	\node[Nonvertex,ellipse,minimum width=0.95cm,text height=2.5cm] at (0,0) {};
    	\node[Nonvertex,draw=none,fill=none,rectangle,text width=2cm,label={\mathstrut{$A'_i$}}] at (0,1.7) {};
    	\node[vertex] (a2) at (0,1.25) {};
    	\node[vertex] (a3) at (0,.75) {};
    	\node[vertex] (a4) at (0,.25) {};
    	\node[vertex] (a5) at (0,-.25) {};
    	\node[vertex] (a6) at (0,-.75) {};
    	\node[vertex] (a7) at (0,-1.25) {};
    \end{scope}
    \begin{scope}[xshift=3cm]
        \node[vertex,label=right:{$w_i$}] (v) at (1.5,0) {};
    	\node[Nonvertex,ellipse,minimum width=0.95cm,text height=2.5cm] at (0,0) {};
    	\node[Nonvertex,draw=none,fill=none,rectangle,text width=2cm,label={\mathstrut{$B'_i$}}] at (0,1.7) {};
    	\node[vertex] (b2) at (0,1.25) {};
    	\node[vertex] (b3) at (0,.75) {};
    	\node[vertex] (b4) at (0,.25) {};
    	\node[vertex] (b5) at (0,-.25) {};
    	\node[vertex] (b6) at (0,-.75) {};
    	\node[vertex] (b7) at (0,-1.25) {};
    \end{scope}
    \draw (a2) -- (b2);
    \draw (a2) -- (b3);
    \draw (a6) -- (b3);
    \draw (a3) -- (b4);
    \draw (a5) -- (b5);
    \draw (a5) -- (b2);
    \draw (a6) -- (b4);
    \draw (a6) -- (b6);
    \draw (a6) -- (b7);
    \draw (a7) -- (b5);
    \draw (a7) -- (b6);
    \foreach \z in {2,3,4,5,6,7}{
        \draw (u) -- (a\z);
        \draw (v) -- (b\z);
    }
    \end{scope}
\end{tikzpicture}
	\caption{An example for the construction used to prove \cref{thmGraphIsomorphismGIComplete}. The ellipses represent independent sets in $G$ and cliques in $G'$.}\label{fig:GI-constr}
\end{figure}

\GraphIsomorphismGIComplete*
\begin{proof}
    We reduce from the isomorphism problem on connected bipartite graphs, which is known to be \GI-complete \cite{uehara2005graph}; note that the authors only claim it for bipartite graphs but the construction ensures that the obtained graph is connected.
    
    Let $G_1,G_2$ be connected bipartite graphs with colour classes $A_1$ and $B_1$, $A_2$ and $B_2$ respectively.
    We construct two graphs $G'_1$ and $G'_2$ as follows.
    The vertex set of $G'_i$ for $i \in \Set{1,2}$ consists of the vertex set of $G_i$ together with two extra vertices $u_i$ and $w_i$. 
    Each of the sets $A_i\cup\{u_i\}$ and $B_i\cup \{w_i\}$ forms a clique in $G'_i$, for $i \in \Set{1,2}$. Furthermore, we add an edge to $G'_i$ between $x\in A_i$ and $y\in B_i$ if and only if $x$ and $y$ are adjacent in $G_i$.
    See \cref{fig:GI-constr} for an illustration of this construction.

    The obtained graphs are clearly cobipartite.
    We prove that they are also $2$-moplex graphs.
    In each of the obtained graphs the two new vertices $u_i$ and $w_i$ form simplicial moplexes (of size one). 
    So what remains to show is that there are no additional moplexes in the constructed graphs.
    Let $a \in A_i$ and suppose it belongs to a moplex $M$ in $G_i'$.
    Since $G_i$ is connected, $a$ has a neighbour in $B_i$, which implies $N[a]\neq N[u_i]$, thus we have that $u_i \notin M$.
    Also, $N(M)$ is a minimal $x,y$-separator for two non-adjacent vertices $x,y \in \V{G'_i}$.
    Since $a \in A_i\cap M$, we have $A_i\cup\{u_i\} \subseteq N[M]$, thus $\Set{x,y} \subseteq B_i \cup \Set{w_i}$.
    However, $B_i \cup \Set{w_i}$ is a clique in $G_i'$, contradicting the fact that $x$ and $y$ are non-adjacent in $G_i'$. By similar arguments no vertex of $B_i$ is moplicial.

    To complete the proof, we show that there is an isomorphism $f: G_1 \to G_2$ if and only if there is an isomorphism $f': G'_1 \to G'_2$.
    Assume there is an isomorphism $f : G_1 \to G_2$. 
    Since $G_1$ and $G_2$ are connected, we must have either $f(A_1) = A_2$ (and then $f(B_1) = B_2)$ or $f(A_1) = B_2$ (and then $f(B_1) = A_2$).
    We extend $f$ to an isomorphism $f': G'_1 \to G'_2$ by 
    setting 
  $$(f'(u_1),f'(w_1)) = \left\{
  \begin{array}{ll} 
  (u_2,w_2),&\hbox{if $f(A_1)=A_2$;}\\
  (w_2,u_2),&\hbox{if $f(A_1) = B_2$.}
  \end{array}
\right.$$

    Now, assume there is an isomorphism $f' : G'_1 \to G'_2$.
    First note that $u_i$ and $w_i$ are the only simplicial vertices in $G'_1$, and thus we have either $f'(u_1) = u_2$ and $f'(w_1) = w_2$, or $f'(u_1) = w_2$ and $f'(w_1) = u_2$.
    This immediately implies that their neighbourhoods are also mapped to each other:
    $f'(A_1) = A_2$ and $f'(B_1) = B_2$, or $f'(A_1) = B_2$ and $f'(B_1) = A_2$.
    Since $f'$ is an isomorphism, for every $x,y \in A_1 \cup B_1$, we have $\Set{x,y} \in \E{G'_1}$ if and only if $\Set{f'(x),f'(y)} \in \E{G'_2}$.
    Using also the fact that $A_i$ and $B_i$ are cliques in $G_i'$ and independent sets in $G_i$, we infer that %
     the restriction of $f'$ to $V(G_1)$ 
     is an isomorphism between $G_1$ and~$G_2$.
\end{proof} 

We note that the proofs of \cref{thmGraphIsomorphismGIComplete,thmMaxCutNPComplete} also imply stronger statements, namely that \textsc{Max-Cut} is \NP-complete and \textsc{Graph Isomorphism} is \GI-complete even for cobipartite graphs with at most $2$ avoidable vertices (recalling~\cref{moplicial implies avoidable}).

\section{Traceability of \texorpdfstring{$2$}{2}-moplex graphs}\label{sec:traceable}
\label{sec:traceability}

A graph is \emph{traceable} if it contains a Hamiltonian path.
It is well-known that every connected proper interval graph is traceable~\cite{MR731128}.
In this section, we generalise this result by proving \cref{thmMtwoHamiltonicity}, i.e.\ that every connected $2$-moplex graph is traceable.

An important ingredient in this section is \emph{Lexicographic Depth First Search} (LDFS), which was introduced in 2008 by Corneil and Krueger~\cite{corneil2008unified}.
Algorithmically, LDFS can be seen as a special case of Depth First Search (DFS) with a ``lexicographic'' tie-breaking rule. 
For our purposes, we use the corresponding (and equivalent) vertex ordering concept (see~\cite[Theorem 2.7]{corneil2008unified}).

\begin{definition}
\label{def:ldfs}
An ordering $\sigma$ of $G$ is an \emph{LDFS ordering} if the following holds: if $ a <_\sigma b <_\sigma c$ and $ac \in E(G)$ and $ab \notin E(G)$, then there exists a vertex $d$ such that $a<_\sigma d <_\sigma b$ and $db \in E(G)$ and $dc\notin E(G)$. 
\end{definition}

A \emph{DFS ordering} of a connected graph $G$ is any ordering of $V(G)$ obtained by a depth first search.
Note that any LDFS ordering is a DFS ordering.
Recall the following properties of any DFS ordering, which we use throughout this section.

\begin{fact}\label{fact neighbours vi}
Let $G$ be a connected graph and let $(v_1,v_2,\dots,v_n)$ be a DFS ordering of $G$.
Then for any $i\in \{2,\dots,n\}$ there exists $j<i$ such that $v_iv_j\in E(G)$.
Also, if $v_iv_{i+1}\notin E(G)$ for some $i\in \{1,2,\dots,n-1\}$, then $i\ge 2$ and $v_iv_{j}\notin E(G)$ for any $j>i$. 
\end{fact}

Recall that the goal of this section is to prove that every connected $2$-moplex graph is traceable.
Our approach relies on the fact that every $2$-moplex graph is a cocomparability graph (as proved in \cref{M2 is cocomp}).
A result of K\"ohler and Mouatadid~\cite{MR3238908} states that every cocomparability graph admits an umbrella-free LDFS ordering.
It is thus natural to consider properties of orderings that are both LDFS and umbrella-free.
In this case, the property from \cref{def:ldfs} can be strengthened as follows.

\begin{lemma}[Corneil et~al.~\cite{corneil2013ldfs}]
    Let $\sigma$ be an umbrella-free LDFS ordering of a cocomparability graph $G$ and let $a,b,c$ be such that 
	$a\LDFS b\LDFS c$, and $ac\in E(G)$, and $ab\notin E(G)$.
	Then, there exists a vertex $d$ such that $a\LDFS d\LDFS b$ and $\Set{a,b,c,d}$  induces a $C_{4}$
	in $G$. \label{enu:Three-point-characterization}
\end{lemma}

The next lemma focuses on the local structure around the non-edges of a connected cocomparability graph $G$ whose endpoints are consecutive in an umbrella-free LDFS ordering.

\begin{lemma}\label{lem:local structure}
Let $G$ be a connected cocomparability graph, let $\sigma  = (v_1,\ldots, v_n)$ be an umbrella-free LDFS ordering of $G$, and let
 $v_iv_{i+1}\notin E(G)$ for some $i\in \{1,\ldots, n-1\}$. Then $v_{i}$ is avoidable in $G$.
\end{lemma}

\begin{proof}
Let $v_{j}$ and $v_{k}$ be two non-adjacent neighbours of $v_{i}$. 
By \cref{fact neighbours vi}, we know that $v_j \LDFS v_i$ and $v_k \LDFS v_i$.
We may assume without loss of generality that $v_{j}\LDFS v_{k}\LDFS v_{i}$.
By \cref{enu:Three-point-characterization}, there exists a vertex $v_\ell$ such that $\{v_i,v_j,v_k,v_\ell\}$ induces a $C_4$ in $G$.
Thus $v_i$ is avoidable.
\end{proof}

Klaus~\cite{MR1251304} introduced another useful way to break ties in a DFS (or any other vertex ordering) procedure, sometimes referred to as a ``$+$ sweep''.
The idea is to prioritise the greatest eligible element with respect to some given ordering.

\begin{definition}\label{def:plus-sweep}
Let $\tau$ be an arbitrary vertex ordering of a graph $G$ (considering the vertices sorted from smallest to largest). Then $\text{(L)DFS}^{+}(\tau)$ is the (L)DFS ordering of $G$ which is lexicographically maximal w.r.t.\ $\tau$ (where the first vertex has highest significance) among all (L)DFS orderings of $G$.
In particular, $\text{(L)DFS}^{+}(\tau)$ always starts with the last vertex of $\tau$.
We provide an example in \cref{fig:DFSp_example}.
\end{definition}

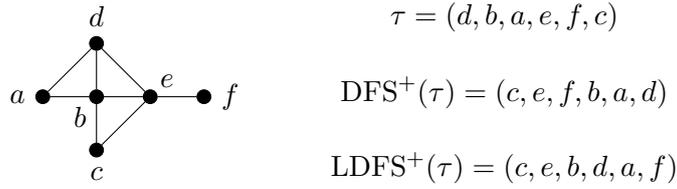
\begin{figure}[ht]
\centering
\begin{tikzpicture}
    \node (left) at (-2,0) {
        \begin{tikzpicture}[rotate=45]
            \tikzstyle{vertex}=[draw,circle,fill,scale=0.5]
            \node[vertex,label=left:{$a$}] (a) at (0,0) {};
            \node[vertex,label={[label distance=-1mm]225:$b$}] (b) at ($(a)+(0.5,-0.5)$) {};
            \node[vertex,label=below:{$c$}] (c) at ($(b)-(0.5,0.5)$) {};
            \node[vertex,label=above:{$d$}] (d) at ($(b)+(0.5,0.5)$) {};
            \node[vertex,label={[label distance=-1mm]45:{$e$}}] (e) at ($(b)+(0.5,-0.5)$) {};
            \node[vertex,label=0:{$f$}] (f) at ($(e)+(0.5,-0.5)$) {};
    
            \draw (a) -- (d);
            \draw (a) -- (b);
            \draw (b) -- (d);
            \draw (e) -- (d);
            \draw (b) -- (c);
            \draw (c) -- (e);
            \draw (e) -- (f);
            \draw (b) -- (e);
        \end{tikzpicture}};
    \node (right) at (3,0) {
        \begin{tikzpicture}
            \node (s1) at (0,0) {$\tau = (d,b,a,e,f,c)$};
            \node (dfsp) at ($(s1)+(0,-1)$) {$\text{DFS}^{+}(\tau) = (c,e,f,b,a,d)$};
            \node (ldfsp) at ($(s1)+(0,-2)$) {$\text{LDFS}^{+}(\tau) = (c,e,b,d,a,f)$};
        \end{tikzpicture}};
\end{tikzpicture}
\caption{A graph and a given ordering $\tau$ together with the orderings $\text{DFS}^{+}(\tau)$ and $\text{LDFS}^{+}(\tau)$.}
\label{fig:DFSp_example}
\end{figure}

We make use of the following important results on LDFS vertex orderings in the class of cocomparability graphs.
\begin{enumerate}[label=(\roman*)]
    \item As already mentioned, K\"ohler and Mouatadid \cite{MR3238908} showed that every cocomparability graph admits an umbrella-free LDFS ordering.
    \item Corneil et al.~\cite{corneil2013ldfs} showed that if $\sigma$ is an umbrella-free ordering of a graph, then so is $\text{LDFS}^{+}(\sigma)$.
    \item Xu and Li \cite{xu2013moplex} showed that every LDFS ordering ends in a moplicial vertex.
\end{enumerate}

These three results imply the following corollary.

\begin{corollary}\label{cor:umbrella-free-LDFS}
    Every cocomparability graph $G$ has an umbrella-free LDFS ordering $\sigma$ such that both first and last vertex are moplicial.
\end{corollary}

\begin{proof}
We first apply the approach of \cite{MR3238908} to obtain an umbrella-free LDFS ordering $\tau$ of $G$ and then we compute $\sigma = \text{LDFS}^{+}(\tau)$.
By~\cite{corneil2013ldfs}, $\sigma$ is an umbrella-free LDFS ordering of $G$. 
Since $\sigma$ starts with the last vertex of $\tau$, it follows from \cite{xu2013moplex} that both first and last vertices of $\sigma$ are moplicial.
\end{proof}

\Cref{moplicial implies avoidable} implies that every graph with at most two avoidable vertices has at most two moplexes.
Using the results developed so far, we can already establish traceability for this special case of $2$-moplex graphs.

\begin{proposition}\label{cor:connected 2avoidable graphs are traceable}
Every connected graph containing at most two avoidable vertices has a Hamiltonian path.
\end{proposition}

\begin{proof}
Let $G$ be a connected graph containing at most two avoidable vertices.
Applying \cref{moplicial implies avoidable,M2 is cocomp}, we infer that $G$ is a cocomparability graph. By \cref{cor:umbrella-free-LDFS}, the graph $G$ has an umbrella-free LDFS ordering $\sigma = (v_1,\ldots, v_n)$ such that vertices $v_1$ and $v_n$ are moplicial.
We claim that $(v_1,\ldots, v_n)$ is a Hamiltonian path in $G$.
Suppose that $v_iv_{i+1}\notin E(G)$ for some $i\in \{1,\ldots, n-1\}$.
Then, by \cref{lem:local structure}, $i\ge 2$ and the vertex $v_i$ is avoidable in $G$.
Since $v_1$ and $v_n$ are moplicial vertices, they are also avoidable, and thus $G$ contains three distinct avoidable vertices, a contradiction. 
\end{proof}

An important ingredient in extending the statement of \cref{cor:connected 2avoidable graphs are traceable} to the whole family of $2$-moplex graphs is the following theorem.

\begin{theorem}[Corneil et~al.~{\cite{corneil2013ldfs}}]
	Let $\sig$ be an umbrella-free LDFS vertex ordering of a cocomparability graph $G$.
	If $G$ admits a Hamiltonian path, then one such path corresponds to $\mathrm{DFS}^{+}(\sig)$.\label{enu:path-cover}
\end{theorem}

We proceed with the proof of the main theorem of this section.

\MtwoHamiltonicity*
\begin{proof}
	Aiming towards a contradiction, fix a connected $2$-moplex graph $G$ of minimum order that does not admit a Hamiltonian path. 
	By \cref{cor:umbrella-free-LDFS,M2 is cocomp}, the graph $G$ has an umbrella-free LDFS ordering $\sig = (v_1,\ldots, v_n)$ such that $v_1$ and $v_n$ are moplicial.
	Note that $n\ge 2$ since otherwise $G$ would have a Hamiltonian path. 

	The minimality of $G$ implies that no two distinct vertices of $G$ have the same closed neighbourhood.
	In particular, every moplex in $G$ is of size one.
	Hence, since $G$ is a \hbox{$2$-moplex} graph, $v_1$ and $v_n$ are the only two moplicial vertices of $G$.
	Since $N[v_1]\neq N[v_n]$, these two vertices belong to distinct moplexes.
	
Note that, as $G$ does not admit a Hamiltonian path, there is at least one vertex that is not adjacent to its immediate successor in $\sigma$.
We next focus on how the removal of any such vertex affects the ordering $\sigma$.
For $i\in \{1,\ldots, n\}$, we denote by $\sig_i$ the ordering of $G-v_i$ obtained from $\sig$ by removing $v_{i}$. 

\begin{claim}\label{lem:adjacent to previous}
    Suppose that $v_iv_{i+1}\notin E(G)$ for some $i\in \{1,\ldots, n-1\}$. 
    Then, $\sig_i$ is an umbrella-free LDFS ordering of $G-v_i$ that retains the property of starting and ending in moplicial vertices.
	Furthermore, we have that $v_{i-1}v_{i}\in E(G)$.
\end{claim}

\begin{proof}
    By \cref{fact neighbours vi}, we have $i\ge 2$.
    Every subsequence of an umbrella-free vertex ordering remains umbrella-free.
	To prove that $\sig_i$ remains an LDFS ordering, observe that, by \cref{fact neighbours vi}, there is no neighbour $v\in N(v_{i})$ such that $v_{i}\LDFS v$. 
    Consequently, for any three vertices $a,b,c$ of $G-v_i$ such that $a<_{\sigma}b<_{\sigma}<c$, $ac \in E(G)$ and $ab \notin E(G)$, any vertex $d$ such that $a<_\sigma d <_\sigma b$, $db \in E(G)$ and $dc\notin E(G)$ is different from $v_i$.
    So all triples satisfying the condition in \Cref{def:ldfs} remain unaffected.
	As $v_n \neq v_i \neq v_1$, the vertex $v_i$ is not moplicial. 
	By \cref{lem:local structure}, $v_i$ is avoidable in $G$, and thus, by \cref{thm:removing_avoidable_vertices}, $v_1$ and $v_n$ are moplicial in $G-v_i$.

	We prove that  $v_{i-1}v_{i}\in E(G)$ by contradiction, so suppose $v_{i-1}v_{i}\notin E(G)$. 
	Now, observe that $N(v_{i})\subseteq N(v_{i-1})$.
	Indeed, for every $v_j\LDFS v_i$ with $v_{j}v_{i}\in E(G)$ and $v_{j}v_{i-1}\notin E(G)$, the vertices $v_{j},v_{i-1}$, and $v_{i}$ form an umbrella.
	But if $N(v_{i})\subseteq N(v_{i-1})$, then 
	$N(v_{i})$ is a minimal $v_{i-1}{,}v_i$-separator in $G$, and thus $\{v_{i}\}$ is a moplex.
	Recall that $i\geq 2$, and thus vertices $v_{1}$, $v_{i}$, and $v_{n}$ belong to three distinct moplexes, a contradiction.
\end{proof}

Considering the ordering obtained by performing a DFS$^+$ sweep on $\sigma$, we define $\DFS $ as the binary relation on the vertex set of $G$ such that for every $u,v\in V(G)$ we have
    \begin{align*}
        u\DFS  v\iff u\text{ is the immediate predecessor of }v\text{ in }\mathrm{DFS}^{+}(\sig).
    \end{align*}
\begin{claim}
	    \label{claim:1}
	    Let $i\in \{1,\ldots, n-1\}$ such that  $v_iv_{i+1}\notin E(G)$, and let $i',i'' $ be such that $v_{i'} \DFS  v_i \DFS  v_{i''}$. %
     Furthermore, assume that $v_n,v_{n-1},\dots,v_{i+1}$ appear before $v_i$ in $\mathrm{DFS}^{+}(\sig)$, and that either 
     \begin{enumerate}
         \item\label[case]{firstcase} $i'<i''=i-1$, or
         \item\label[case]{secondcase} $i''<i'=i-1$ and $N(v_{i'})\cap \{v_{i'-1},v_{i'-2},\dots,v_{i''+1}\}=\emptyset$.
     \end{enumerate}
    Then $v_{i'}v_{i''}\not\in E(G)$.
\label{enu:extendHamPath}
	\end{claim}

	\begin{proof}
    Note first that $\deg(v_i)\ge 2$, since otherwise $G$ would contain three distinct moplexes $\{v_1\}$, $\{v_i\}$ and $\{v_n\}$.
    This immediately implies that $v_{i''}$ is adjacent to $v_i$.

    We claim that $v_{i'}$ is the neighbour of the vertex $v_i$ that is visited first by DFS$^+(\sigma)$.
    Let $v^*$ be the neighbour of the vertex $v_i$ that is visited first by DFS$^+(\sigma)$.
    Due to \cref{fact neighbours vi} applied to $\sigma$, every neighbour $v_j$ of $v_i$ satisfies $j<i$.
    In particular, $v^*<_\sigma v_i$.
    Since DFS$^+(\sigma)$ is prioritising the choice of the maximal non-visited vertex with respect to $\sigma$ (see \cref{def:plus-sweep}), the assumption that 
     $\{v_n,v_{n-1}\dots,v_{i+1}\}$ appear before $v_i$ in $\mathrm{DFS}^{+}(\sig)$
    implies that $v^*\DFS v_i$.
    Consequently $v_{i'} = v^*$, as claimed.
    In particular, $v_{i'}$ is adjacent to $v_i$.
    Furthermore, since $v_iv_{i'},v_iv_{i''}\in E(G)$, by \cref{fact neighbours vi}, 
      the vertices $v_{i'}$ and $v_{i''}$ appear before $v_i$ in the ordering $\sigma$, that is, $i',i''< i$.
    Suppose that $v_{i'}v_{i''}\in E(G)$ (see \cref{fig:claim1} for an illustration).
		By \cref{lem:adjacent to previous}, we have that $\sig_{i}$ is an umbrella-free LDFS ordering of $G-v_{i}$.
       Furthermore, since $v_{i'}$ is the neighbour of $v_i$ that is visited first by DFS$^+(\sigma)$, the prefix of DFS$^+(\sigma_i)$ until the vertex $v_{i'}$ coincides with the prefix of DFS$^+(\sigma)$ until the same vertex.
       Next, observe that in both \cref{firstcase,secondcase} the properties of DFS$^+$ search and the assumption that $v_{i'}v_{i''}\in E(G)$ imply that the vertex $v_{i'}$ is the immediate predecessor of $v_{i''}$ in $\mathrm{DFS}^{+}(\sig_i)$.
       
        By \cref{lem:local structure}, the vertex $v_i$ is avoidable in $G$.
		Hence $v_i$ is not a cut-vertex, that is, $G-v_i$ is connected.
		Furthermore, since $v_i$ is not moplicial, \cref{remove_avoidable_vertices} implies that  $G-v_i$ is a $2$-moplex graph.  
		Now, observe that by the minimality of $G$, the graph $G-v_i$ contains a Hamiltonian path, and thus by  \cref{enu:path-cover},	
		\begin{align}
		\mathrm{DFS}^{+}(\sig_{i})=(v_{n},\dots,v_{i'},v_{i''},\dots)\label{eq:Claim2HP}
		\end{align}
		corresponds to a Hamiltonian path of $G-v_{i}$.
        To conclude the proof it is enough to extend \cref{eq:Claim2HP} to a Hamiltonian path $(v_{n},\dots,v_{i'},v_{i},v_{i''},\dots)$ in $G$, contradicting the choice of $G$.
        Therefore, $v_{i'}v_{i''}\notin E(G).$
	\end{proof}
	\begin{figure}[ht]
		    \centering
            \begin{tikzpicture}
				\tikzstyle{vertex}=[draw,circle,fill,scale=0.5]
				\tikzstyle{prededge}=[violet,very thick,->,-latex]
				\tikzstyle{prededgePoss}=[violet,very thick,dashed,->,-latex]
				\tikzstyle{knownedge}=[]
				\tikzstyle{knownNONedge}=[dashed]
				\def\dist{2.7cm}
				\node (center) at (0,0) {};
				\node[vertex] (vi) at (center) {};
				\node[vertex] (x) at ($(vi)-(\dist,0)$) {};
				\node (xRest) at ($(x)-(\dist/2,0)$) {$\cdots$};
				\node[vertex] (y) at ($(x)-(\dist,0)$) {};
				\node[vertex] (vii) at ($(vi)+(\dist,0)$) {};
				\node (viiRest) at ($(vii)+(\dist/2,0)$) {$\cdots$};
				\node (final) at ($(vii)+(\dist,0)$) {};
				\node (yRest) at ($(y)-(\dist/2,0)$) {$\cdots$};

				\draw[knownNONedge] (vi) edge (vii);

				\draw[prededge] (vi) to (x);
				\draw[prededge,bend left,in=145,out=30] (y) to (vi);
                \draw[prededge,bend right,out=325,in=230] ($(vi)+(45:1)$) to (y);
				\draw[very thick,myRed,bend right,out=330,in=210] (y) edge (x);

				\node (predLabel) at ($(center)+(-0.4,0.6)$) {\textcolor{violet}{$\mathbf{\DFS}$}};
				\node (yLabel) at ($(y)-(0,0.4)$) {$v_{i'}$}; %
				\node (xLabel) at ($(x)-(0,0.6)$) {$v_{i''} = v_{i-1}$}; %
				\node (viLabel) at ($(vi)-(0,0.4)$) {\textcolor{darkgray}{$v_i$}};
				\node (viiLabel) at ($(vii)-(0,0.4)$) {\textcolor{darkgray}{$v_{i+1}$}};
			\end{tikzpicture}
            \centering
			\begin{tikzpicture}
				\tikzstyle{vertex}=[draw,circle,fill,scale=0.5]
				\tikzstyle{prededge}=[violet,very thick,->,-latex]
				\tikzstyle{prededgePoss}=[violet,very thick,dashed,->,-latex]
				\tikzstyle{knownedge}=[]
				\tikzstyle{knownNONedge}=[dashed]
				\def\dist{2.7cm}
				\node (center) at (0,0) {};
				\node[vertex] (vi) at (center) {};
				\node[vertex] (x) at ($(vi)-(\dist,0)$) {};
				\node (xRest) at ($(x)-(\dist/2,0)$) {$\cdots$};
				\node[vertex] (y) at ($(x)-(\dist,0)$) {};
				\node[vertex] (vii) at ($(vi)+(\dist,0)$) {};
				\node (viiRest) at ($(vii)+(\dist/2,0)$) {$\cdots$};
				\node (final) at ($(vii)+(\dist,0)$) {};
				\node (yRest) at ($(y)-(\dist/2,0)$) {$\cdots$};

				\draw[knownNONedge] (vi) edge (vii);

				\draw[prededge] (x) to (vi);
				\draw[prededge,bend right,out=325,in=210] (vi) to (y);
                \draw[prededge,bend right,out=325,in=210] ($(vi)+(45:1)$) to (x);
				\draw[very thick,myRed,bend right,out=330,in=210] (y) edge (x);

				\node (predLabel) at ($(center)+(-0.4,0.6)$) {\textcolor{violet}{$\mathbf{\DFS}$}};
				\node (yLabel) at ($(y)-(0,0.4)$) {$v_{i''}$}; %
				\node (xLabel) at ($(x)-(0,0.6)$) {$v_{i'} = v_{i-1}$}; %
				\node (viLabel) at ($(vi)-(0,0.4)$) {\textcolor{darkgray}{$v_i$}};
				\node (viiLabel) at ($(vii)-(0,0.4)$) {\textcolor{darkgray}{$v_{i+1}$}};
			\end{tikzpicture}
			\caption{The vertices are drawn in order of $\sig$.
			Edges are solid while non-edges are dashed.
			The purple, directed edges stand for the $\DFS$-relation; they are solid as we know them to be edges.
			\Cref{claim:1} proves that $v_{i'}v_{i''}$ is not an edge: the upper part shows \cref{firstcase} and the lower part \cref{secondcase}.}
		    \label{fig:claim1}
		\end{figure}
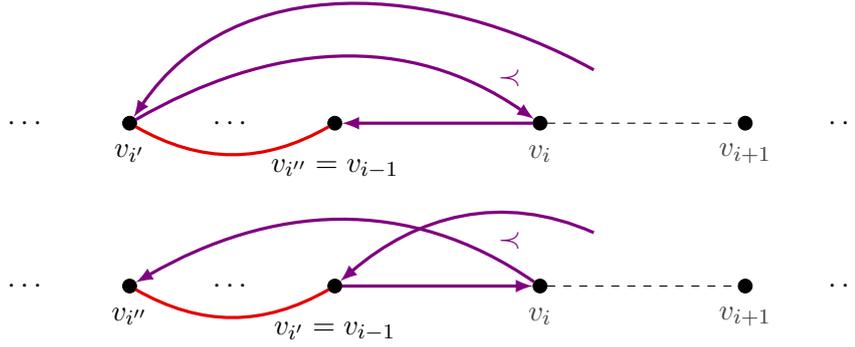
		
    Now fix $k$ to be the maximal integer in $\{2,\ldots, n-1\}$ such that $v_{k}v_{k+1}\notin E(G)$, and let $G'=G-v_k$.
    Applying \cref{lem:adjacent to previous} yields that $\sig_k =(v_{1},v_{2},\dots,v_{k-1},v_{k+1},\dots,v_{n})$ is an umbrella-free LDFS ordering of $G'$ starting and ending on moplicial vertices.
    
    Since $v_k$ is an avoidable vertex in $G$ it is not a cut-vertex, and hence $G'$ is connected. Furthermore, since $v_k$ is not moplicial, \cref{remove_avoidable_vertices} implies that $G'$ is a $2$-moplex graph.
    Thus, by the minimality of $G$ and by \cref{enu:path-cover}, it follows that $\mathrm{DFS}^{+}(\sig_k)$ corresponds to a Hamiltonian path of $G-v_{k}$.
    At this point, we know that $\mathrm{DFS}^{+}(\sig)$ and $\mathrm{DFS}^{+}(\sig_k)$ share the prefix 
    $(v_{n}, v_{n-1},\dots, v_{k+1})$, and thus $v_{n}\DFS v_{n-1}\DFS \dots\DFS v_{k+1}$.
    Since $v_{k}v_{k+1}\notin E(G)$, \cref{fact neighbours vi} implies that $v_{k+1}$ has at least one neighbour in $\Set{v_1,\dots,v_{k-1}}$, and hence $v_{k+1}$ must be adjacent in $G$ to its successor in $\mathrm{DFS}^{+}(\sig)$.
    
    \begin{claim}
        \label{claim:2}
	    We have $v_{k+1}\DFS  v_{k-1}$.
    \end{claim}
    \begin{proof}
        Let $j$ be the largest integer such that $j\le k$ and $v_{j}v_{k+1}\in E(G)$.
	    The choice of $j$ implies that $v_{k+1}\DFS v_{j}$ and $j\neq k$.
    	
    	Suppose towards a contradiction that $j \neq k-1$, that is, $v_{k+1}v_{k-1} \notin \E{G}$ (see \cref{fig:claim2} for an illustration).
    	This implies that $v_{j}$ is adjacent to both $v_{k-1}$ and $v_{k}$, as otherwise 
     either $\{v_j,v_{k-1},v_{k+1}\}$ or $\{v_j,v_k,v_{k+1}\}$ form an umbrella in $\sigma$.
    	By \cref{lem:adjacent to previous}, we also
	    have $v_{k}v_{k-1}\in E(G)$, which 
      implies $v_{k+1}\DFS  v_{j}\DFS  v_{k}\DFS  v_{k-1}$.
 As $v_jv_{k-1}\in \E{G}$ this 
     contradicts \cref{enu:extendHamPath} (applied to $i=k$, with assumption \ref{firstcase}).
	    \begin{figure}[ht]
	        \centering
			\begin{tikzpicture}
				\tikzstyle{vertex}=[draw,circle,fill,scale=0.5]
				\tikzstyle{prededge}=[violet,very thick,->,-latex]
				\tikzstyle{prededgePoss}=[violet,very thick,dashed,->,-latex]
				\tikzstyle{knownedge}=[]
				\tikzstyle{knownNONedge}=[dashed]
				\def\dist{2.7cm}
				\node (center) at (0,0) {};
				\node[vertex] (vk) at (center) {};
				\node[vertex] (kv) at ($(vk)-(\dist,0)$) {};
				\node[vertex,myRed] (vj) at ($(kv)-(\dist,0)$) {};
				\node[vertex] (vkk) at ($(vk)+(\dist,0)$) {};
				\node (final) at ($(vkk)+(\dist,0)$) {};
				\node (dvj) at ($(vj)-(\dist/2,0)$) {$\cdots$};
				\node (dkv) at ($(kv)-(\dist/2,0)$) {$\cdots$};
				\node (dfinal) at ($(final)-(\dist/2,0)$) {$\cdots$};

				\draw[knownNONedge] (vk) edge (vkk);
				\draw[knownNONedge,myRed,very thick,bend left] (kv) edge (vkk);

				\draw[prededge,bend right,out=320,in=220] (vkk) to (vj);
				\draw[prededge,bend left,out=330,in=210] (vj) to (vk);
				\draw[prededge] (vk) to (kv);
				\draw[draw,very thick,bend left,out=20,in=160] (vj) edge (kv);
				\node (predLabel) at ($(vkk)+(-0.4,0.6)$) {\textcolor{violet}{$\mathbf{\DFS}$}};
				\node (vjLabel) at ($(vj)-(0,0.4)$) {\textcolor{myRed}{$v_j$}};
				\node (vkLabel) at ($(vk)+(0,0.4)$) {\textcolor{darkgray}{$v_k$}};
				\node (kvLabel) at ($(kv)-(0,0.4)$) {\textcolor{darkgray}{$v_{k-1}$}};
				\node (vkkLabel) at ($(vkk)-(0,0.4)$) {\textcolor{darkgray}{$v_{k+1}$}};
			\end{tikzpicture}
		\caption{The vertices are drawn in order of $\sig$.
			Edges are solid while non-edges are dashed.
			The purple directed edges stand for the $\DFS$-relation; they are solid as we know them to be edges.
			\Cref{claim:2} proves that $\mathrm{DFS}^{+}(\sig)$ visits $v_{k-1}$ after $v_{k+1}$, as a different vertex $v_j$ would yield a contradiction to \cref{claim:1}.}
		    \label{fig:claim2}
		\end{figure}
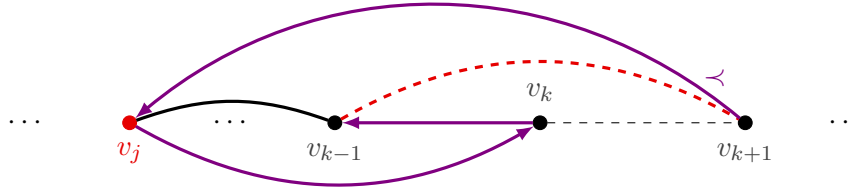
    \end{proof}

    From $v_kv_{k-1} \in \E{G}$ and the definition of $\mathrm{DFS}^{+}$ it follows that $v_k$ is visited immediately after $v_{k-1}$, i.e.\ $v_{k-1}\DFS v_k$.
    So by now we have established $v_{n}\DFS \dots\DFS  v_{k+1}\DFS  v_{k-1}\DFS  v_{k}$.
    This corresponds to a path of length $n-k+1$ in $G$, so in the following we assume $k>2$.

    \begin{claim}
        \label{claim:3}
	    We have that $v_{k-2}v_{k-1}\notin E(G)$.
    \end{claim}
    \begin{proof}
	    Suppose $v_{k-2}v_{k-1}$ is an edge (see \cref{fig:claim3} for an illustration).
	    Then $v_{k-2}v_{k}\notin E(G)$,
	    as otherwise $v_{k+1}\DFS  v_{k-1}\DFS  v_{k}\DFS  v_{k-2}$, which contradicts \cref{enu:extendHamPath} (with $i=k$ under assumption \ref{secondcase}).
	    However, since $v_{k-2}v_{k}\notin E(G)$, observe that $N(v_{k})\subseteq N(v_{k-2})$, as otherwise $v_{k-2}$ and $v_{k}$ would form an umbrella in $\sig$ with any neighbour $x$ of $v_k$ that is not a neighbour of $v_{k-2}$, using the fact that $x \LDFS v_{k-2}$, by \cref{fact neighbours vi} (with $i=k$), and $x\neq v_{k-1}$.
	    The inclusion $N(v_{k})\subseteq N(v_{k-2})$ implies that $\{v_{k}\}$ is a moplex.
	    Since $v_1$ and $v_n$ belong to distinct moplexes in $G$ other than $\{v_k\}$, this contradicts that $G$ is a $2$-moplex graph.
    \end{proof}
    \begin{figure}[ht]
	        \centering
			\begin{tikzpicture}
				\tikzstyle{vertex}=[draw,circle,fill,scale=0.5]
				\tikzstyle{prededge}=[violet,very thick,->,-latex]
				\tikzstyle{prededgePoss}=[violet,very thick,dashed,->,-latex]
				\tikzstyle{knownedge}=[]
				\tikzstyle{knownNONedge}=[dashed]
				\def\dist{2.4cm}
				
				\pgfdeclarelayer{background}
			
			    \pgfsetlayers{background,main}
			    \begin{pgfonlayer}{main}
				\node (center) at (0,0) {};
				\node[vertex] (vk) at (center) {};
				\node[vertex] (kv) at ($(vk)-(\dist,0)$) {};
				\node[vertex] (kkv) at ($(kv)-(\dist,0)$) {};
				\node (dkkv) at ($(kkv)-(\dist/2,0)$) {$\cdots$};
				\node[vertex] (vumbrella) at ($(kkv)-(\dist,0)$) {};
				\node (du) at ($(vumbrella)-(\dist/2,0)$) {$\cdots$};
				\node[vertex] (vkk) at ($(vk)+(\dist,0)$) {};
				\node (final) at ($(vkk)+(\dist,0)$) {};
				\node (dfinal) at ($(final)-(\dist/2,0)$) {$\cdots$};

				\draw[knownNONedge] (vk) edge (vkk);

				\draw[prededge,bend right,out=-60,in=-120] (vkk) to (kv);
				\draw[prededge] (kv) edge (vk);
				\draw[draw,very thick,myRed] (kkv) edge (kv);
				\draw[very thick,bend right,out=-50,in=-125,dash pattern=on 4pt off 3pt] (vk) to (kkv);
				\draw[very thick,bend right,out=-40,in=-140,dash pattern=on 4pt off 3pt] (kkv) edge (vumbrella);
				\draw[very thick,bend left,out=70,in=110] (vumbrella) edge (vk);

				\node (predLabel) at ($(vkk)+(-0.4,1.1)$) {\textcolor{violet}{$\mathbf{\DFS}$}};
				\node (vjLabel) at ($(kkv)-(0,0.5)$) {\textcolor{darkgray}{$v_{k-2}$}};
				\node (vkLabel) at ($(vk)-(0,0.5)$) {\textcolor{darkgray}{$v_k$}};
				\node (kvLabel) at ($(kv)-(0,0.5)$) {\textcolor{darkgray}{$v_{k-1}$}};
				\node (vkkLabel) at ($(vkk)-(0,0.5)$) {\textcolor{darkgray}{$v_{k+1}$}};
				\node (xLabel) at ($(vumbrella)-(0,0.5)$) {\textcolor{darkgray}{$x$}};
				\node (umbrellaLabel) at ($(kkv)+(0.2,1.4)$) {umbrella};%
				\end{pgfonlayer}
				\begin{pgfonlayer}{background}
				    \begin{scope}
				        \path[clip]
                            (vumbrella) to [bend left,out=70,in=110] (vk) to [bend right,out=-50,in=-125]  (kkv.center) to [bend right,out=-40,in=-140] (vumbrella) to cycle;
                        \fill[opacity=0.1] (vumbrella) rectangle ($(vk)+(0,2.2)$);
				    \end{scope}
				\end{pgfonlayer}
			\end{tikzpicture}
			\caption{The vertices are drawn in order of $\sig$.
			Edges are solid while non-edges are dashed.
			The purple, directed edges stand for the $\DFS$-relation; they are solid as we know them to be edges.
			\Cref{claim:3} shows that $v_{k-2}v_{k-1}$ is a non-edge as otherwise $G$ contains another moplex.}
		    \label{fig:claim3}
		\end{figure}
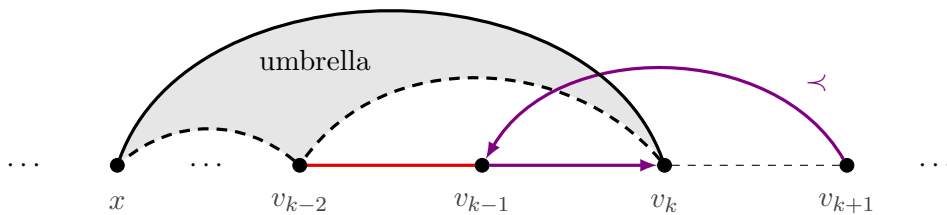
		
    Since $v_{k-2}v_{k-1}\notin E(G)$, we can use \cref{fact neighbours vi,lem:adjacent to previous} to obtain $k\geq4$ and $v_{k-2}v_{k-3}\in E(G)$, respectively.
    As $v_{k}$ is not moplicial, it has degree at least two.
    Fix $v_{j}$ to be its neighbour such that the value of $j<k-1$ is maximised, i.e.\ $v_{k}\DFS  v_{j}$.
    We now show that no value of $j$ is realisable.
    
    As we know that $v_{k-2}v_{k-1}\notin E(G)$, and since $v_{k-2}v_{k}\in E(G)$ would violate $\sigma $ being a DFS ordering (see \cref{fact neighbours vi}), we deduce that $v_{k-2}v_{k}\notin E(G)$, and thus we have $j\neq k-2$.
    
    Suppose next that $j=k-3$, or equivalently $v_{k-3}v_{k}\in E(G)$.
    Then $v_{k-1}v_{k-3}\in E(G)$, otherwise the vertices $v_{k-3}\LDFS v_{k-1} \LDFS v_{k}$ violate \cref{def:ldfs}.
    Since $v_{k-1}\DFS  v_{k}\DFS  v_{k-3}$, \cref{enu:extendHamPath} (with $i=k$ under assumption \ref{secondcase})
    implies $v_{k-1}v_{k-3}\notin E(G)$, a contradiction.

    Finally, suppose that $j< k-3$ and observe that $v_{k-2}v_{j}$ and $v_{k-3}v_{j}$ are edges, as otherwise the corresponding two vertices would form an umbrella on $\sig$ with $v_{k}$.
    In particular, we have
    \begin{align*}
        v_{k+1}\DFS  v_{k-1}\DFS  v_{k}\DFS  v_{j}\DFS  v_{k-2}\DFS  v_{k-3}\,.
    \end{align*}
    By \cref{enu:extendHamPath} (with $i=k-2$ under \ref{firstcase}), we get $v_jv_{k-3}\notin\E{G}$, a contradiction. 
    This concludes the proof of \cref{thmMtwoHamiltonicity}.
\end{proof}

We conclude the section by mentioning that, due to \cite{MR3238908}, computing a minimum path cover can be done in linear time for cocomparability graphs.
This result, along with \cref{thmMtwoHamiltonicity,M2 is cocomp}, implies the following.

\begin{corollary}
Given a connected $2$-moplex graph $G$, a Hamiltonian path in $G$ can be computed in linear time.
\end{corollary}

\section{Concluding remarks} 

Moplexes provide a tool that has the potential to lift the beneficial structural properties of simplicial modules in chordal graphs to the setting of all graphs, see e.g.~\cite{BerryBBS10,berry2001asteroidal,BerryB01}.
We thus believe that graphs with a bounded number of moplexes form interesting graph classes which were well overdue for further study.
We introduce the moplex number of a graph, focusing our study on properties of graphs with moplex number $2$, the smallest nontrivial class in the moplex-number hierarchy.

In \cref{thm:removing_avoidable_vertices}, we identify a graph operation which preserves the moplexes, namely, the removal of an avoidable non-moplicial vertex.
One can easily determine other such operations (e.g.\ removing a universal vertex or a true twin).
It would be interesting to characterise the class of $k$-moplex graphs by identifying a set of operations and a base class that can be used to generate every member of the class.
The existence of such a characterisation is still open, even when $k=2$.

We show in \cref{thmMtwoHamiltonicity} that every connected $2$-moplex graph is traceable.
In this respect, we conjecture the following strengthening of \cref{thmMtwoHamiltonicity}.
(In fact, no counterexample is known even for the stronger property of pancyclicity.)
\begin{conjecture}
Every $2$-connected $2$-moplex graph has a Hamiltonian cycle.
\end{conjecture}

Some of the questions answered for the case $k=2$ can also be asked for $k>2$.
For instance, how do the classes of $k$-moplex graphs relate with the hierarchy of hereditary graph classes?
Also, what is the complexity of \textsc{Clique}, \textsc{Clique Cover}, and \textsc{Colouring} for $k$-moplex graphs?

\bibliographystyle{plain}
\bibliography{literature}

\end{document}